\documentclass[reqno]{amsart}
\pdfoutput=1
\usepackage{amssymb}
\usepackage[dvips]{epsfig}
\usepackage{graphicx}
\usepackage{color}
\usepackage{amsmath}
\usepackage{amssymb}
\usepackage{amsfonts}
\usepackage{amsthm}
\usepackage{bbm}
\usepackage{tikz}
\usepackage[colorlinks=true,linkcolor=blue,citecolor=brown]{hyperref}

\newtheorem{thm}{Theorem}[section]
\newtheorem{lem}[thm]{Lemma}

\theoremstyle{remark}
\newtheorem{rem}{\bf Remark}[section]
\theoremstyle{definition}

\numberwithin{equation}{section}

\begin{document}
\title[The Existence of Full Dimensional KAM tori for NLS]{The Existence of Full Dimensional KAM tori for Nonlinear Schr\"odinger equation}

\author[H. Cong]{Hongzi Cong}
\address[H. Cong]{School of Mathematical Sciences, Dalian University of Technology, Dalian 116024, China}
\email{conghongzi@dlut.edu.cn}


\date{\today}



\begin{abstract}
In this paper, we will prove the existence of full dimensional tori for 1-dimensional nonlinear Schr\"odinger equation  with periodic boundary conditions
\begin{equation*}\label{L1}
\mathbf{i}u_t-u_{xx}+V*u+\epsilon|u|^4u=0,\hspace{12pt}x\in\mathbb{T},
\end{equation*}
where $V*$ is the convolution potential. Here the radius of the invariant torus satisfies a slower decay, i.e.
\begin{equation*}\label{031601}
I_n\sim e^{- \ln^{\sigma}|n|},\qquad \mbox{as}\ |n|\rightarrow\infty,
\end{equation*}
for any $\sigma>2$, which improves the result given by Bourgain (J. Funct. Anal. 229 (2005), no.1, 62-94).

\end{abstract}

\maketitle

\section{Introduction and main results}
The nonlinear Schr\"{o}dinger
equation is a typical Hamiltonian partial differential equation which has drawn a lot of concerns during the last decades. It is well studied by many authors about the existence and linear stability of invariant tori, long time stability and growth of Sobolev norm of the solutions by some techniques coming from infinite dimensional Hamiltonian system.

An open problem was raised by Kuksin (see Problem 7.1 in \cite{Kuksin2004}):
 \vskip8pt
 {\it Can the full dimensional KAM tori be expected with a suitable decay for Hamiltonian partial differential equations, for example, $$I_n\sim|n|^{-C}$$ with some $ C>0$ as $|n|\rightarrow+\infty$? }

 In 2005, Bourgain \cite{Bourgain2005JFA} constructed the full dimensional tori for the following nonlinear Schr\"{o}dinger
equation \begin{equation}\label{L1}
\mathbf{i}u_t-u_{xx}+V*u+\epsilon|u|^4u=0,\hspace{12pt}x\in\mathbb{T},
\end{equation}where the radius of the invariant torus satisfies
\begin{equation}\label{031701}
I_n\sim e^{-\sqrt{n}},\qquad \mbox{as}\ |n|\rightarrow \infty.
\end{equation}
Our basic motivation is to improve Bourgain's result in the sense that there exist full dimensional tori satisfying \begin{equation}\label{031601}
I_n\sim e^{- \ln^{\sigma}|n|},\qquad \mbox{as}\ |n|\rightarrow\infty,
\end{equation}with any $\sigma>2$ for equation (\ref{L1}), which is closer to the conjecture by Kuksin.

In Fourier modes $(q_n)_{n\in\mathbb{Z}}$,
(\ref{L1}) can be rewritten as
\begin{equation*}\label{L3}
\dot{q}_n=\mathbf{i}\frac{\partial{H}}{\partial\bar{q}_n}
\end{equation*}
with the Hamiltonian
\begin{equation}\label{L4}
H(q,\bar{q})=\sum_{n\in\mathbb{Z}}\left(n^2+V_n\right)|q_n|^2+\epsilon\sum_{n_1-n_2+n_3-n_4+n_5-n_6=0}q_{n_1}\bar{q}_{n_2}q_{n_3}\bar{q}_{n_4}q_{n_5}\bar{q}_{n_6}
\end{equation}
and $(V_n)_{n\in\mathbb{Z}}$ are independently chosen in $[0,1]$.
To state our result, we introduce Diophantine conditions which is firstly given by Bourgain in \cite{Bourgain2005JFA}. For $x\in\mathbb{R}$, denote $\|x\|=\inf\limits_{y\in\mathbb{Z}}|x-y|$. Then we say a vector $V\in[0,1]^{\mathbb{Z}}$ is Diophantine, if there exists a real number $\gamma>0$ such that
\begin{equation}\label{S1}
\left|\left|\sum_{n\in\mathbb{Z}}l_nV_n\right|\right|\geq \gamma\prod_{n\in\mathbb{Z}}\frac{1}{1+l_n^2\langle n\rangle^4},
\end{equation}
\mbox{for any $l\in\mathbb{Z}^{\mathbb{Z}}$ with $0<\#\mathrm{supp}\ l<\infty$}, where
\begin{equation*}
\mathrm{supp}\ l=\left\{n:l_n\neq0\right\}
\end{equation*}
and
\begin{equation*}
\langle n\rangle =\max\{1,|n|\}.
\end{equation*}

Now  our main result is as follows:
\begin{thm}\label{031930}
Given any $\sigma>2$ and a frequency vector $\omega=(\omega_n)_{n\in\mathbb{Z}}$ satisfying the Diophantine condition (\ref{S1}), then there exists a small $\epsilon_*(\sigma,\gamma)>0$ depending on $\sigma$ and $\gamma$ only, and  for any $0<\epsilon<\epsilon_*(\sigma,\gamma)$ there exists $V\in[0,1]^{\mathbb{Z}}$  such that equation (\ref{L1}) has a full dimensional invariant torus $\mathcal{E}$  satisfying:\\
(1). the amplitude $I=(I_n)_{n\in\mathbb{Z}}$ of $\mathcal{E}$ restricted as
\begin{equation}
\frac{1}4e^{-2\ln^{\sigma}\lfloor n\rfloor}\leq |I_n|\leq 4e^{-2\ln^{\sigma}\lfloor n\rfloor},\qquad \forall\ n;
\end{equation}
(2). the frequency on $\mathcal{E}$ prescribed to be $(n^2+\omega_n)_{n\in\mathbb{Z}}$;\\
(3). the invariant torus $\mathcal{E}$ is linearly stable.
\end{thm}

\begin{rem}
From Lemma 4.1 in \cite{Bourgain2005JFA}, we know the following resonance issue:
\begin{equation*}\label{S3}
P\left\{V: \left|\left|\sum_{n\in\mathbb{Z}}l_nV_n\right|\right|< \gamma\prod_{n\in\mathbb{Z}}\frac{1}{1+l_n^2\langle n\rangle^4},\forall\  l\neq0 \ \mathrm{with}\ \#\mathrm{supp}\ l<\infty\right\}<C\gamma,
\end{equation*}
where $P$ is the standard probability measure on $[0,1]^{\mathbb{Z}}$ and $C>0$ is an absolute constant.
\end{rem}

Since 1990's, the KAM theory for infinite dimensional Hamiltonian system has been well developed to study the existence and linear stability of invariant tori for Hamiltonian PDEs. See \cite{BBM2014,BBM2016,BBHM2018,CW1993,CY2000,K1987,K2000,Kuksin2004,KP1996,LY2010,P'1996,P1996,W1990} for the related works for 1-dimensional PDEs. For high dimensional PDEs, the situation becomes more complicated due to the multiple eigenvalues of Laplacian operator. Bourgain \cite{Bourgain1998Annals,Bourgain2005} developed a new method initialed by Craig-Wayne \cite{CW1993} to prove the existence of KAM tori for $d$-dimensional nonlinear Schr$\ddot{\mbox{o}}$dinger equations (NLS) and $d$-dimensional NLW with $d\geq 1$, based on the Newton iteration, Fr$\ddot{\mbox{o}}$hlich-Spencer techniques, Harmonic analysis and semi-algebraic set theory. This is so-called C-W-B method.
Later, Eliasson-Kuksin \cite{EK2010} obtained both the existence and the linear stability of KAM tori for $d$-dimensional NLS in a classical KAM way. Also see \cite{BB2012,BLP2015,BWJEMS,W2020} for the related problem.

In the above works, the obtained KAM tori are of low (finite) dimension which are the support of the quasi-periodic solutions. It must be pointed out that the constructed quasi-periodic solutions are not typical in the sense that the low dimensional tori have measure zero for any reasonable measure on the infinite dimensional phase space. It is natural at this point to find the full dimensional tori which are the support of the almost periodic solutions. The first result on the existence of almost periodic solutions for Hamiltonian PDEs was proved by Bourgain in \cite{Bourgain1996} using C-W-B method. Later, P$\ddot{\mbox{o}}$schel \cite{Poschel2002} (also see \cite{GX2013} by Geng-Xu) constructed the almost periodic solutions for 1-dimensional NLS by the classical KAM method. The basic idea to obtain these almost periodic solutions is by perturbing the quasi-periodic ones. That is why the action $I=(I_n)$ must satisfy some very strong compactness properties. In fact, the following super-exponential decay for the action $I$ is given
 $$I_n\sim \varepsilon^{2^{|n|}},\qquad \varepsilon\ll 1$$
 as $n\rightarrow\infty$. It means that these solutions are with very high regularity and look like the quasi-periodic ones.

The first try to obtain the existence of full dimensional tori with slower decay was given by Bourgain \cite{Bourgain2005JFA}, who proved that 1-dimensional NLS has a full dimensional KAM torus of prescribed frequencies with the actions of the tori obeying the estimate (\ref{031701}). {{A big difference from \cite{Bourgain1996} and \cite{Poschel2002}, Bourgain treated all Fourier modes at once under Diophantine conditions (\ref{S1}), which is totally different from those in constructing low dimensional KAM tori and causes a much worse small denominator problem. It is well known that one of cores of KAM theory is the dealing of these small divisors. }}.

To overcome such a problem, Bourgain firstly made use of two simple but important facts: let
$n_i$ be a finite set of modes satisfying
\begin{equation*}
|n_1|\geq |n_2|\geq |n_3|\geq \cdots
\end{equation*}
 and
 \begin{equation}\label{022001}
n_1-n_2+n_3-n_4+\cdots=0.
\end{equation}
 One observation is that the following inequality holds
\begin{equation*}
\sum_{i\geq1}\sqrt{|n_i|}-2\sqrt{|n_1|}\geq\frac14\sum_{i\geq 3}{\sqrt{|n_i|}};
\end{equation*}the other one is that:
the first two biggest indices $|n_1|$ and $|n_2|$ can be controlled by other indices unless $n_1=n_2$, i.e.
\begin{equation}\label{022004}|n_1|+|n_2|\leq C\left(|n_3|+|n_4|+\cdots\right),
 \end{equation}under some further assumptions
 \begin{equation}\label{022002}
n_1^2-n_2^2+n_3^2-n_4^2+\cdots=O(1).
\end{equation}
The conditions (\ref{022001}) and (\ref{022002}) hold true for NLS, which follow from momentum conservation
and the quadratic growth of the frequencies for 1-dimensional NLS (also whenever the small divisors appear). Now it is possible to deal with the small divisor introduced by (\ref{S1}). The second key step is some quantitative analysis are needed to guarantee the KAM iteration works.

Recently, Cong-Liu-Shi-Yuan \cite{CLSY2018JDE} generalized Bourgain's result from $\theta=1/2$ to any $0<\theta<1$ in a classical KAM way, where the actions of the tori satisfying
\begin{equation*}\label{020203}I_n\sim e^{-{|n|^{\theta}}},\qquad \theta\in (0,1).
\end{equation*}
The authors also proved the obtained tori are stable in a sub-exponential long time. Another important progress is given by Biasco-Massetti-Procesi \cite{BMP2021Poincare}, who proved the existence
and linear stability of almost periodic solution for 1-dimensional NLS by constructing a
rather abstract counter-term theorem for infinite dimensional Hamiltonian system. The idea is based on a more
geometric point of view. What is more interesting is that a construction of elliptic tori independent of their dimension
is proved by this method, which gives a new perspective to study the existence of invariant tori with finite dimension and infinite dimension.

When infinite systems of coupled harmonic oscillators with finite-range couplings are considered, a much better result has been obtained. Precisely, P$\ddot{\mbox{o}}$schel in \cite{Poschel1990} proved the existence of full dimensional tori for infinite dimensional Hamiltonian system with spatial structure of short range couplings consisting of
connected sets only, where the action satisfies
\begin{equation}\label{032301}I_n\sim e^{-\ln^{\sigma}|n|},\qquad n\rightarrow \infty
\end{equation}with any $\sigma>2$. In particular, in the simplest case, i.e. only nearest neighbour coupling, the result can be optimized for any $\sigma>1$.

Of course, Hamiltonian PDEs are not of short range, and do not contain such spatial structure. Therefore it is curious to ask the following question:
{\it do there exist full dimensional tori satisfying (\ref{032301}) for Hamiltonian PDEs such as equation (\ref{L1})?} In this paper, we will give an affirmative answer. Now we will introduce some basic idea in the below. The proof of main theorem follows from the frame given by \cite{CLSY2018JDE}. In the first step, it is proven a so-called tame inequality (\ref{001}) in Lemma \ref{005}. Combining (\ref{022004}) and (\ref{001}), it affords an opportunity to control the small divisor (\ref{S1}). Next, one needs a cut off procedure such that one deals with the tori of finite dimension in the $s$-th step KAM iteration. When constructing the tori with a slower decay, the cut off becomes larger. Roughly speaking, following Bourgain's idea the dimension $N_s$ of the tori in \cite{Bourgain2005JFA} is given by
\begin{equation}\label{032402}
N_s\sim B_s^2,
\end{equation}
while in our case it should be
\begin{equation*}
N_s\sim \exp\left\{B_s^{\frac1{\sigma}}\right\},
\end{equation*}
(see (\ref{032401}) for the details). Obviously, one has
\begin{equation}\label{032404}
\exp\left\{B_s^{\frac1{\sigma}}\right\}\gg B_s^2,
\end{equation}
which is too large.
By a delicate calculation, we find that it only needs the following estimate
\begin{equation*}
\sum_{i=1}^{N_s}\sqrt{i}\sim B_s
\end{equation*}
in Bourgain's case, which implies
\begin{equation}\label{032403}
N_s\sim B_{s}^{\frac23}.
\end{equation}
There may be no essential difference between (\ref{032402}) and (\ref{032403}). But this simple observation is very useful in our case. Precisely,
\begin{equation*}
\sum_{i=1}^{N_s}\ln^{\sigma}|i|\sim B_s
\end{equation*}
means
\begin{equation*}
N_s\sim \frac{B_s}{\ln^{\sigma-1}B_s},
\end{equation*}
which is an important improvement than (\ref{032404}).
Fortunately, the low bound of the small divisors still can be controlled by super-exponential decay using $\sigma>2$ (see (\ref{022402}) for the details). The similar discussions also appear in other places such as the estimates of Poisson bracket, the Hamiltonian flow and the shift of the frequency. Overcome the above difficulties, we finish the proof of our main result.

Finally we have some further comments between the result given by P$\ddot{\mbox{o}}$schel in \cite{Poschel1990} and ours. The main difference is that the infinite dimensional  Hamiltonian systems of short range are studied in \cite{Poschel1990} while the nonlinear Schr\"odinger equation are not, which is in our case. So our result may be optimal in the sense that we get the same conclusion as in \cite{Poschel1990} (see (\ref{032301})).  Besides this, P$\ddot{\mbox{o}}$schel constructed a KAM theorem in action-angle coordinates while we study in Cartesian coordinates. There are some differences between two coordinates. See \cite{BFN2020} for more comments. On the other hand the Diophantine conditions (\ref{S1}), which is introduced by Bourgain \cite{Bourgain2005JFA} firstly, is more explicit and direct compared with the Diophantine conditions given by P$\ddot{\mbox{o}}$schel in \cite{Poschel1990} where some fixed approximation functions are needed. Based on the above two points, our method could be applied to some other Hamiltonian PDEs such as nonlinear wave equation combining with some techniques given in \cite{CY2021}.

\section{The Norm of the Hamiltonian}
In this section, we will introduce some notations and definitions firstly.

For any $\sigma>2$, define the Banach space $\mathfrak{H}_{\sigma,\infty}$ of all complex-valued sequences
$q=(q_n)_{n\in\mathbb{Z}}$ with
\begin{equation}\label{042501}
\|q\|_{\sigma,\infty}=\sup_{n\in\mathbb{Z}}\left|q_n\right|e^{\ln^{\sigma}\lfloor n\rfloor}<\infty,
\end{equation}
where
\begin{equation*}\label{030401}
\lfloor n \rfloor=\max\{c(\sigma),|n|\}
\end{equation*}with $c(\sigma)>1$ is a constant depending on $\sigma$ only.
\begin{rem}The constant $c(\sigma)$ is given explicitly by (\ref{030402}) in Lemma \ref{122203}. Let
\begin{equation*}
\lfloor n\rfloor'=\max\left\{2,|n|\right\},
\end{equation*}
and we can define the Banach space $\mathfrak{H}'_{\sigma,\infty}$ of all complex-valued sequences
$q=(q_n)_{n\in\mathbb{Z}}$ by
\begin{equation*}
\|q\|'_{\sigma,\infty}=\sup_{n\in\mathbb{Z}}\left|q_n\right|e^{\ln^{\sigma}\lfloor n\rfloor'}<\infty.
\end{equation*}
It is easy to see that the two Banach space $\mathfrak{H}_{\sigma,\infty}$ and $\mathfrak{H}'_{\sigma,\infty}$ are equivalent in the sense that there two positive constants $c_1(\sigma)$ and $c_2(\sigma)$ depending on $\sigma$ only such that
\begin{equation*}
c_1(\sigma)\|\cdot\|\leq \|\cdot\|'\leq c_2(\sigma)\|\cdot\|.
\end{equation*}
Actually, for fixed $\sigma>2$, one has
\begin{equation*}
\left|q_n\right|e^{\ln^{\sigma}\lfloor n\rfloor}=\left|q_n\right|e^{\ln^{\sigma}\lfloor n\rfloor'}
\end{equation*}if $|n|$ is large enough (i.e. $|n|\geq c(\sigma)$). The main reason we choose the Banach space $\mathfrak{H}_{\sigma,\infty}$ instead of $\mathfrak{H}'_{\sigma,\infty}$ is to simplify the proof of some technical lemmas and some estimates in KAM iteration.
\end{rem}
Introduce the notations $I_n=|q_n|^2$ and $J_n=I_n-I_n(0)$, where $I_n(0)$ will be considered as the initial data and satisfy
\begin{equation}\label{M13}
I_{n}(0)\leq e^{-2\ln^{\sigma}\lfloor n\rfloor}.
\end{equation} Assume the Hamiltonian $R(q,\bar q)$ with the form of
\begin{equation*}
R(q,\bar q)=\sum_{a,k,k'\in\mathbb{N}^{\mathbb{Z}}}B_{akk'}\mathcal{M}_{akk'},
\end{equation*}
where the so-called monomials
\begin{equation*}\label{030502}
\mathcal{M}_{akk'}=\prod_{n\in\mathbb{Z}}I_n(0)^{a_n}q_n^{k_n}\bar q_n^{k_n'},
\end{equation*}
satisfy the mass conservation
\begin{equation}
\sum_{n\in\mathbb{Z}}\left(k_n-k_n'\right)=0,
\end{equation}
and the momentum conservation
\begin{equation}\label{050901}
\sum_{n\in\mathbb{Z}} \left(k_n-k_n'\right)n=0
\end{equation}
and $B_{akk'}$ are the coefficients.
Given a monomial $\mathcal{M}_{akk'}$,
define by
\begin{equation*}
\mbox{supp}\ \mathcal{M}_{akk'}=\{n\in\mathbb{Z}:a_n+k_n+k_n'\neq 0\},
\end{equation*}
and
\begin{equation*}
\mbox{degree}\ \mathcal{M}_{akk'}=2|a|+|k|+|k'|:=\sum_{n\in\mathbb{Z}}(2a_n+k_n+k_n').
\end{equation*}
\begin{rem}
Here we always assume that $$\mbox{degree}\ \mathcal{M}_{akk'}\geq 4.$$ When equation (\ref{L1}) is considered, one has $$\mbox{degree}\ \mathcal{M}_{akk'}\geq 6.$$
\end{rem}

Denote by
\begin{equation*}
n_1^*=\max\{|n|:a_n+k_n+k_n'\neq 0\}.
\end{equation*}Now we will define the weighted $\ell^{\infty}$ norm of the Hamiltonian $R(q,\bar q)$ with the weight $\rho\geq0$ by
\begin{def}\label{083103}
\begin{equation}\label{042602}
||R||_{\sigma,\rho}=\sup_{a,k,k'\in\mathbb{N}^{\mathbb{Z}}}\frac{|B_{akk'}|}{e^{\rho\left(\sum_{n\in\mathbb{Z}}\left(2a_n+k_n+k_n'\right)\ln^{\sigma} \lfloor n\rfloor-2\ln^{\sigma} \lfloor n_1^*\rfloor\right)}}.
\end{equation}
\end{def}
\begin{rem}
Here we point out that the index $\sigma$ will be fixed but the index $\rho$ will become larger during KAM iteration (see the details at the beginning of Section \ref{031501}). For simplicity we will write $||R||_{\rho}=||R||_{\sigma,\rho}$ in the below.
\end{rem}
For any $a,k,k'$, denote $(n^*_i)_{i\geq1}$ the decreasing rearrangement of
\begin{equation*}
\{|n|:\ \mbox{where $n$ is repeated}\ 2a_n+k_n+k_n'\ \mbox{times}\}.
\end{equation*}
If $k$ and $k'$ satisfy (\ref{050901}),
then it is easy to see that
\begin{equation*}
\sum_{n\in\mathbb{Z}}\left(2a_n+k_n+k_n'\right)\ln^{\sigma} \lfloor n\rfloor-2\ln^{\sigma} \lfloor n_1^*\rfloor\geq 0.
\end{equation*}
In fact we can prove a positive lower bound of
$$\sum_{n\in\mathbb{Z}}\left(2a_n+k_n+k_n'\right)\ln^{\sigma} \lfloor n\rfloor-2\ln^{\sigma}\lfloor n_1^*\rfloor,$$
which is important to overcome the small divisor. Precisely we have the following lemma:
\begin{lem}\label{005}

Given any $a,k,k'$, assume the condition (\ref{050901}) is satisfied.
Then for any $\sigma>2$, one has
\begin{equation}\label{001}
\sum_{n\in\mathbb{Z}}(2a_n+k_n+k_n'){\ln^{\sigma}\lfloor n\rfloor}-2{\ln^{\sigma}\lfloor n_1^*\rfloor}\geq\frac12\sum_{i\geq 3}{\ln^{\sigma}\lfloor n_i^*\rfloor}.
\end{equation}
\end{lem}
\begin{proof}Without loss of generality,
denote $(n_i)_{i\geq 1}$ the system $\{|n|:\ \mbox{where $n$ is repeated}\ 2a_n+k_n+k_n'\ \mbox{times}\}$ which satisfies $\ |n_1|\geq |n_2|\geq\cdots$. In view of (\ref{050901}), there exist $(\mu_i)_{i\geq1}$ with $\mu_i\in\{\pm1\}$ such that
\begin{equation*}
\sum_{i\geq1}\mu_in_i=0.
\end{equation*}
Hence we have
\begin{equation*}
|n_1|\leq \sum_{i\geq 2}|n_i|
\end{equation*}
and
\begin{equation*}
\lfloor n_1\rfloor\leq \sum_{i\geq 2}\lfloor n_i\rfloor.
\end{equation*}
Consequently
\begin{equation*}
{\ln^{\sigma}\lfloor n_1\rfloor}\leq \ln^{\sigma}\left({\sum_{i\geq 2}\lfloor n_i\rfloor}\right).
\end{equation*}
Noting that
\begin{equation*}
\sum_{n\in\mathbb{Z}}(2a_n+k_n+k_n'){\ln^{\sigma}\lfloor n\rfloor}={\sum_{i\geq 1}\ln^{\sigma}\lfloor n_i\rfloor},
\end{equation*}
 thus the inequality (\ref{001}) will follow from the inequality
\begin{equation}\label{002}
\sum_{i\geq 2}{\ln^{\sigma}\lfloor n_i\rfloor}\geq\ln^{\sigma}\left({\sum_{i\geq 2}\lfloor n_i\rfloor}\right)+\frac12\sum_{i\geq 3}{\ln^{\sigma}\lfloor n_i\rfloor}.
\end{equation}
Using (\ref{122201}) in Lemma \ref{122203}, one has
\begin{equation}\label{030101}
\ln^{\sigma}x+\ln^{\sigma}y\geq \ln^{\sigma}(x+y)+\frac12\ln^{\sigma}y, \quad \mbox{for}\ x\geq y\geq c(\sigma).
\end{equation} In view of (\ref{030101}) and by iteration, one obtains
\begin{eqnarray*}
\sum_{i\geq 2}{\ln^{\sigma}\lfloor n_i\rfloor}
&=&{\ln^{\sigma}\lfloor n_2\rfloor}+{\ln^{\sigma}\lfloor n_3\rfloor}+\sum_{i\geq 4}{\ln^{\sigma}\lfloor n_i\rfloor}\\
&\geq&\ln^{\sigma}\left(\sum_{2\leq i\leq3}\lfloor n_i\rfloor\right)+\frac12{\ln^{\sigma}\lfloor n_3\rfloor}+\sum_{i\geq 4}{\ln^{\sigma}\lfloor n_i\rfloor}\\
&=&\ln^{\sigma}\left(\sum_{2\leq i\leq3}\lfloor n_i\rfloor\right)+{\ln^{\sigma}\lfloor n_4\rfloor}+\frac12{\ln^{\sigma}\lfloor n_3\rfloor}+\sum_{i\geq 5}{\ln^{\sigma}\lfloor n_i\rfloor}\\
&\geq&\ln^{\sigma}\left(\sum_{2\leq i\leq4}\lfloor n_i\rfloor\right)+\frac12\left(\sum_{3\leq i\leq 4}{\ln^{\sigma}\lfloor n_i\rfloor}\right)+\sum_{i\geq 5}{\ln^{\sigma}\lfloor n_i\rfloor}\\
&&\cdots\\
&\geq&\ln^{\sigma}\left(\sum_{ i\geq2}\lfloor n_i\rfloor\right)+\frac12\left(\sum_{i\geq3}{\ln^{\sigma}\lfloor n_i\rfloor}\right),
\end{eqnarray*}
which finishes the proof of (\ref{002}).
\end{proof}

\begin{lem}(\textbf{Poisson Bracket})\label{010}
Let $\sigma>2,\rho>0$ and $$0<\delta_1,\delta_2<\min\left\{\frac14\rho,3-2\sqrt{2}\right\}.$$ Then one has
\begin{equation}\label{042704}
\left\|\{R_1,R_2\}\right\|_\rho\leq \frac{1}{\delta_2}\exp\left\{\frac{1000}{\delta_1}\cdot\exp\left\{\left(\frac{100}{\delta_1}\right)^{\frac1{\sigma-1}}\right\}\right\}\left\|R_1\right\|_{\rho-\delta_1}\left\|R_2\right\|_{\rho-\delta_2}.
\end{equation}
\end{lem}
\begin{proof}
Let
\begin{equation*}
R_1(q,\bar q)=\sum_{a,k,k'\in\mathbb{N}^{\mathbb{Z}}} b_{akk'}\mathcal{M}_{akk'}
\end{equation*}
and
\begin{equation*}
R_2(q,\bar q)=\sum_{A,K,K'\in\mathbb{N}^{\mathbb{Z}}} B_{AKK'}\mathcal{M}_{AKK'}.
\end{equation*}
Then one has
\begin{equation*}
\{R_1,R_2\}=\sum_{a,k,k',A,K,K'\in\mathbb{N}^{\mathbb{Z}}}b_{akk'}B_{AKK'}\{\mathcal{M}_{akk'},\mathcal{M}_{AKK'}\},
\end{equation*}
where
\begin{eqnarray*}
\{\mathcal{M}_{akk'},\mathcal{M}_{AKK'}\}
&=&{\textbf{i}}\sum_{j\in\mathbb{Z}}\left(\prod_{n\neq j}I_n(0)^{a_n+A_n}q_n^{k_n+K_n}\bar{q}_n^{k_n'+K_n'}\right)\\
&&\times\left((k_jK_j'-k_j'K_j)I_j(0)^{a_j+A_j}q_j^{k_j+K_j-1}\bar{q}_j^{k_j'+K_j'-1}\right).
\end{eqnarray*}
The coefficient of the monomial $\mathcal{M}_{\alpha\kappa\kappa'}$ in $\left\{R_1,R_2\right\}$ is given by
\begin{equation}\label{006}
 B_{\alpha\kappa\kappa'}=\textbf{i}\sum_{j\in\mathbb{Z}}\sum_{*}\sum_{**}(k_jK_j'-k_j'K_j)b_{akk'}B_{AKK'},
\end{equation}
where
\begin{equation*}
\sum_{*}=\sum_{a,A\in\mathbb{N}^{\mathbb{Z}} \atop a+A=\alpha},
\end{equation*}
and
\begin{equation*}
\sum_{**}=\sum_{k,k',K,K'\in\mathbb{N}^{\mathbb{Z}}\atop \mbox{when}\ n\neq j, k_n+K_n=\kappa_n,k_n'+K_n'=\kappa_n';\mbox{when}\ n=j, k_n+K_n-1=\kappa_n,k_n'+K_n'-1=\kappa_n'}.
\end{equation*}
To estimate (\ref{042704}), some simple facts are given firstly:

$\textbf{1}.$ If $j\notin\ \mbox{supp}\ (k+k') \bigcap\ \mbox{supp}\ (K+K')$, then
\begin{equation*}
k_jK_j'-k_j'K_j=0.
\end{equation*}
Hence we always assume $j\in\ \mbox{supp}\ (k+k') \bigcap\ \mbox{supp}\ (K+K')$, which implies
\begin{equation}\label{122801}
|j|\leq \min\{n_1^*,N^*_1\}.
\end{equation}

 Let
\begin{equation*}
\nu_1^*=\max\{|n|:\alpha_n+\kappa_n+\kappa_n'\neq0\},
\end{equation*}and then the following inequality always holds
\begin{equation}\label{041801}
\nu_1^*\leq \max\{n_1^*,N_1^*\}.
\end{equation}
Combining (\ref{122801}) and (\ref{041801}), one has
\begin{equation}\label{031901}
\ln^{\sigma}\lfloor j\rfloor+\ln^{\sigma}\lfloor\nu_1^*\rfloor-\ln^{\sigma}\lfloor n_1^*\rfloor-\ln^{\sigma}\lfloor N_1^*\rfloor\leq 0.
\end{equation}

$\textbf{2}.$ Note that
\begin{eqnarray}
\nonumber\sum_{i\geq 1}\ln^{\sigma}\lfloor n_i^*\rfloor
\nonumber&=&\sum_{n\in\mathbb{Z}}(2a_n+k_n+k_n')\ln^{\sigma}\lfloor n\rfloor\\
\nonumber&\geq&\sum_{n\in\mathbb{Z}}(2a_n+k_n+k_n')\\
\label{042604}&\geq&\sum_{n\in\mathbb{Z}}(k_n+k_n')
\end{eqnarray}
and
\begin{eqnarray}
\nonumber\sum_{i\geq 3}\ln^{\sigma}\lfloor n_i^*\rfloor
\nonumber&\geq&\sum_{n\in\mathbb{Z}}(2a_n+k_n+k_n')-2\\
\nonumber&\geq&\frac12\sum_{n\in\mathbb{Z}}(2a_n+k_n+k_n')\\
\label{042605}&\geq&\frac12\sum_{n\in\mathbb{Z}}(k_n+k_n').
\end{eqnarray}
Based on (\ref{042604}) and (\ref{042605}), we obtain
\begin{eqnarray}\label{042606}
\sum_{n\in\mathbb{Z}}(k_n+k_n')(K_n+K_n')
&\leq& 2\left(\sum_{i\geq 1}\ln^{\sigma}\lfloor n_i^*\rfloor\right)\left(\sum_{i\geq 3}\ln^{\sigma}\lfloor N_i^*\rfloor\right).
\end{eqnarray}

$\textbf{3}.$ It is easy to see that
\begin{eqnarray}
&&\nonumber\sum_{n\in\mathbb{Z}}(2\alpha_n+\kappa_n+\kappa_n')\\
&=&\label{052802}\left(\sum_{n\in\mathbb{Z}}\left(2a_n+k_n+k_n'\right)\right)+\left(\sum_{n\in\mathbb{Z}}\left(2A_n+K_n+K_n'\right)\right)-2
\end{eqnarray}
and
\begin{eqnarray}
&&\nonumber\sum_{n\in\mathbb{Z}}(2\alpha_n+\kappa_n+\kappa_n')\ln^{\sigma}\lfloor n\rfloor\\
&=&\left(\sum_{n\in\mathbb{Z}}\left(2a_n+k_n+k_n'\right)\ln^{\sigma}\lfloor n\rfloor\right)\nonumber\\
&&+\left(\sum_{n\in\mathbb{Z}}\left(2A_n+K_n+K_n'\right)\ln^{\sigma}\lfloor n\rfloor\right)-2\ln^{\sigma}\lfloor j\rfloor\label{052801}.
\end{eqnarray}
In view of (\ref{042602}) and (\ref{001}) in Lemma \ref{005}, one has
\begin{equation}
|b_{akk'}|
\label{007}\leq||R_1||_{\rho-\delta_1}e^{\rho\left(\sum_{n\in\mathbb{Z}}(2a_n+k_n+k_n')
\ln^{\sigma}\lfloor n\rfloor-2\ln^{\sigma}\lfloor n_1^*\rfloor\right)}e^{-\frac12\delta_1
\left(\sum_{i\geq 3}\ln^{\sigma}\lfloor n_i^*\rfloor\right)},
\end{equation}
and
\begin{equation}
|B_{AKK'}|\leq \label{008}||R_2||_{\rho-\delta_2}e^{\rho\left(\sum_{n\in\mathbb{Z}}(2A_n+K_n+K_n')
\ln^{\sigma}\lfloor n\rfloor-2\ln^{\sigma}\lfloor N_1^*\rfloor\right)}e^{-\frac12\delta_2
\left(\sum_{i\geq 3}\ln^{\sigma}\lfloor N_i^*\rfloor\right)}.
\end{equation}
 Using (\ref{052801}), substitution of (\ref{007}) and (\ref{008}) in (\ref{006}) gives
\begin{eqnarray*}
\label{009}\left|B_{\alpha\kappa\kappa'}\right|
\nonumber&\leq&||R_1||_{\rho-\delta_1}||R_2||_{\rho-\delta_2}\sum_{j\in\mathbb{Z}}\sum_{*}\sum_{**}\left|k_jK_j'-k_j'K_j\right|\\
&&\times e^{\rho\left(\sum_{n\in\mathbb{Z}}\left(2a_n+k_n+k_n'\right)
\ln^{\sigma}\lfloor n\rfloor-2\ln^{\sigma}\lfloor n_1^*\rfloor\right)}\\
&&\times e^{\rho\left(\sum_{n\in\mathbb{Z}}\left(2A_n+K_n+K_n'\right)
\ln^{\sigma}\lfloor n\rfloor-2\ln^{\sigma}\lfloor N_1^*\rfloor\right)}\\
\nonumber&&\times e^{-\frac12\delta_1
\left(\sum_{i\geq 3}\ln^{\sigma}\lfloor n_i^*\rfloor\right)}e^{-\frac12\delta_2
\left(\sum_{i\geq 3}\ln^{\sigma}\lfloor N_i^*\rfloor\right)}\\
\nonumber&=&||R_1||_{\rho-\delta_1}||R_2||_{\rho-\delta_2}e^{\rho\left(\sum_{n\in\mathbb{Z}}
\left(2\alpha_n+\kappa_n+\kappa_n'\right)\ln^{\sigma}\lfloor n\rfloor -2\ln^{\sigma}\lfloor \nu_1^*\rfloor\right)}\\
&&\times\sum_{j\in\mathbb{Z}}\sum_{*}\sum_{**}\left|k_jK_j'-k_j'K_j\right| e^{2\rho\left(\ln^{\sigma}\lfloor j\rfloor+\ln^{\sigma}\lfloor\nu_1^*\rfloor-\ln^{\sigma}\lfloor n_1^*\rfloor-\ln^{\sigma}\lfloor N_1^*\rfloor\right)}\\
&&\times e^{-\frac12\delta_1
\left(\sum_{i\geq 3}\ln^{\sigma}\lfloor n_i^*\rfloor\right)}e^{-\frac12\delta_2
\left(\sum_{i\geq 3}\ln^{\sigma}\lfloor N_i^*\rfloor\right)}.
\end{eqnarray*}
To show (\ref{042704}) holds, it suffices to prove
\begin{equation}\label{041802}
I\leq \frac{1}{\delta_2}\exp\left\{\frac{1000}{\delta_1}\cdot\exp\left\{\left(\frac{100}{\delta_1}\right)^{\frac1{\sigma-1}}\right\}\right\},
\end{equation}
where
\begin{eqnarray*}
I&=&\sum_{j\in\mathbb{Z}}\sum_{*}\sum_{**}\left|k_jK_j'-k_j'K_j\right| e^{2\rho\left(\ln^{\sigma}\lfloor j\rfloor+\ln^{\sigma}\lfloor\nu_1^*\rfloor-\ln^{\sigma}\lfloor n_1^*\rfloor-\ln^{\sigma}\lfloor N_1^*\rfloor\right)}\\
&&\times e^{-\frac12\delta_1
\left(\sum_{i\geq 3}\ln^{\sigma}\lfloor n_i^*\rfloor\right)}e^{-\frac12\delta_2
\left(\sum_{i\geq 3}\ln^{\sigma}\lfloor N_i^*\rfloor\right)}.
\end{eqnarray*}

Now we will prove the inequality (\ref{041802}) holds in the following two cases:

 $\textbf{Case. 1.} \ \nu_1^*\leq N_1^*$.

$\textbf{Case. 1.1.}\ |j|\leq n_3^*\Rightarrow\lfloor j\rfloor\leq \lfloor n_3^*\rfloor.$

 Using $0<\delta_1\leq \frac14 \rho$, one has
\begin{eqnarray}
e^{2\rho\left(\ln^{\sigma}\lfloor j\rfloor-\ln^{\sigma}\lfloor n_1^*\rfloor\right)}
\leq e^{2\rho\left(\ln^{\sigma}\lfloor n_3^*\rfloor-\ln^{\sigma}\lfloor n_1^*\rfloor\right)}
\leq e^{\frac12\delta_1\left(\ln^{\sigma}\lfloor n_3^*\rfloor-\ln^{\sigma}\lfloor n_1^*\rfloor\right)}\label{051001},
\end{eqnarray}
and then
\begin{eqnarray}
e^{2\rho\left(\ln^{\sigma}\lfloor j\rfloor+\ln^{\sigma}\lfloor \nu_1^*\rfloor-\ln^{\sigma}\lfloor n_1^*\rfloor-\ln^{\sigma}\lfloor N_1^*\rfloor\right)}e^{-\frac12\delta_1\left(\sum_{i\geq 3}\ln^{\sigma}\lfloor n_i^*\rfloor\right)}
\leq e^{-{\frac16\delta_1}\left(\sum_{i\geq 1}\ln^{\sigma}\lfloor n_i^*\rfloor\right)}\label{042603},
\end{eqnarray}
which follows from (\ref{041801}) and (\ref{051001}).

Note that if $j,a,k,k'$ are specified, and then $A,K,K'$ are uniquely determined.
In view of (\ref{042606}) and (\ref{042603}), we have
\begin{eqnarray*}
I
&\leq&\sum_{a,k,k'\in\mathbb{N}^{\mathbb{Z}}}2\left(\sum_{i\geq 1}\ln^{\sigma}\lfloor n_i^*\rfloor\right)\left(\sum_{i\geq 3}\ln^{\sigma}\lfloor N_i^*\rfloor\right)e^{-\frac16\delta_1\left(\sum_{i\geq 1}\ln^{\sigma}\lfloor n_i^*\rfloor\right)}e^{-\frac12\delta_2\left(\sum_{i\geq 3}\ln^{\sigma}\lfloor N_i^*\rfloor\right)}\\
&\leq&\frac{48}{e^2\delta_1\delta_2}\sum_{a,k,k'\in\mathbb{N}^{\mathbb{Z}}}
e^{-\frac1{12}\delta_1\left(\sum_{i\geq 1}\ln^{\sigma}\lfloor n_i^*\rfloor\right)}\qquad \mbox{(in view of (\ref{042805*}) with $p=1$)}\\
&=&\frac{48}{e^2\delta_1\delta_2}\sum_{a,k,k'\in\mathbb{N}^{\mathbb{Z}}}
e^{-\frac1{12}\delta_1\left(\sum_{n\in\mathbb{Z}}\left(2a_n+k_n+k'_n\right)\ln^{\sigma}\lfloor n\rfloor\right)}\\
&\leq&\frac{48}{e^2\delta_1\delta_2}\left(\sum_{a\in\mathbb{N}^{\mathbb{Z}}}
e^{-\frac1{12}\delta_1\sum_{n\in\mathbb{Z}}2a_n\ln^{\sigma}\lfloor n\rfloor}\right)
\left(\sum_{k\in\mathbb{N}^{\mathbb{Z}}}e^{-\frac1{12}\delta_1\sum_{n\in\mathbb{Z}}k_n\ln^{\sigma}\lfloor n\rfloor}\right)^2
\\
&\leq&\frac{48}{e^2\delta_1\delta_2}\prod_{n\in\mathbb{Z}}\left({1-e^{-\frac16\delta_1\ln^{\sigma}\lfloor n\rfloor}}\right)^{-1}
\left({1-e^{-\frac1{12}\delta_1\ln^{\sigma}
\lfloor n\rfloor}}\right)^{-2}      \\
&&\nonumber\mbox{(which is based on Lemma \ref{a3})}\\
&\leq&\frac{48}{e^2\delta_1\delta_2}\left(\exp\left\{{\frac{216}{\delta_1}\cdot \exp\left\{{\left(\frac{48}{\delta_1} \right)}^{\frac{1}{\sigma-1}}\right\}}\right\}\right)^{3}\ \ \mbox{(in view of (\ref{122401}))}\\
&\leq&\frac{1}{\delta_2}\exp\left\{{\frac{1000}{\delta_1}\cdot \exp\left\{{\left(\frac{48}{\delta_1} \right)}^{\frac{1}{\sigma-1}}\right\}}\right\}.
\end{eqnarray*}

$\textbf{Case. 1.2.}\ |j|>n_3^*\Rightarrow|j|\in\{n_1^*,n_2^*\}.$

If $2a_j+k_j+k'_j>2$, then $|j|\leq n_3^*$, we are in $\textbf{Case. 1.1.}$. Hence in what follows, we always assume
\begin{equation*}
2a_j+k_j+k'_j\leq2,
\end{equation*}
which implies
\begin{equation}\label{2.24}
k_j+k_j'\leq 2.
\end{equation}
Denote $(n_i)_{i\geq1}$ by the system $\left\{n:\mbox{where $n$ is repeated $2a_n+k_n+k_n'$ times}\right\}$. Without loss of generality, we assume that $n_1=n_1^*$ and $n_2=n_2^*$ below.
From (\ref{031901}), (\ref{2.24}) and $j\in\{n_1,n_2\}$, it follows that
\begin{equation*}
I\nonumber
\nonumber\leq2\sum_{a,k,k'\in\mathbb{N}^{\mathbb{Z}}}\left(K_{n_1}+K_{n_1}'+K_{n_2}+K_{n_2}'\right) \cdot e^{-\frac12\left(\delta_1\left(\sum_{i\geq 3}\ln^{\sigma}\lfloor n_i^*\rfloor\right){+\delta_2\left(\sum_{i\geq 3}\ln^{\sigma}\lfloor N_i^*\rfloor\right)}\right)}.
\end{equation*}
In view of (\ref{042605}) and  (\ref{052802}), we have
\begin{eqnarray}
\sum_{n\in\mathbb{Z}}\left(2\alpha_n+\kappa_n+\kappa'_n\right)
\leq\label{042701}2\left(\sum_{i\geq3}\ln^{\sigma}\lfloor n_{i}^*\rfloor\right)+2\left(\sum_{i\geq3}\ln^{\sigma}\lfloor N_{i}^{*}\rfloor\right).
\end{eqnarray}
Moreover, note that  $\forall j$,
\begin{equation}\label{030520}
 K_j+K'_j\leq \kappa_j+\kappa'_j-k_j-k'_j+2\leq\kappa_j+\kappa'_j+2.
\end{equation}
Using (\ref{042701}) and (\ref{030520}), one has
\begin{eqnarray}
\nonumber I &\leq& 2\sum_{a,k,k'\in\mathbb{N}^{\mathbb{Z}}}\left(\kappa_{n_1}+\kappa'_{n_1}+\kappa_{n_2}+\kappa'_{n_2}+4\right)  \\
&&\label{2.26}\times e^{-\frac14\delta_1\left(\sum_{i\geq 3}\ln^{\sigma}\lfloor n_i^*\rfloor\right)}e^{-\frac18\delta\sum_{n\in\mathbb{Z}}(2\alpha_n+\kappa_n+\kappa'_n)},
\end{eqnarray}
where $\delta=\min\{\delta_1,\delta_2\}$.

\begin{rem}\label{042703} Firstly, note that $\{n_1,n_2\}\cap \mathrm{supp}\ \mathcal{M}_{\alpha\kappa\kappa'}\neq \emptyset$. Thus $n_1$  (or $n_2$) ranges in a set of cardinality no more than \begin{equation}\label{042702}\#\mathrm{supp} \ \mathcal{M}_{\alpha\kappa\kappa'}\leq\sum_{n\in\mathbb{Z}}(2\alpha_n+\kappa_n+\kappa'_n).
\end{equation}Secondly, if $\{n_i\}_{i\geq 3}$ and $n_1$ (resp. $n_2$) is specified, then $n_2$ (resp. $n_1$) is determined uniquely. Thirdly, if $\{n_i\}_{i\geq 1}$ is given, then $\{2a_n+k_n+k'_n\}_{n\in\mathbb{Z}}$ is specified, and hence $(a,k,k')$ is specified up to a factor of
\begin{equation}\label{030203}
\prod_{n\in\mathbb{Z}}\left(1+l_n^2\right),
\end{equation}
where
$$l_n=\#\{j:n_j=n\}.$$
\end{rem}
Since $n_1\geq n_2>n_3^*$, one has
\begin{equation}\label{030202}
\prod_{n\in\mathbb{Z}\atop n=n_1,n_2}\left(1+l_n^2\right)\leq 5.
\end{equation}
Following (\ref{2.26})-(\ref{030202}), we thus obtain
\begin{eqnarray}
\nonumber I
\nonumber
\nonumber&\leq& 60\sum_{\{n_i\}_{i\geq3}}\prod_{m\in\mathbb{Z}\atop |m|\leq n_3^*}\left(1+l_m^2\right)e^{-\frac14\delta_1\left(\sum_{i\geq 3}\ln^{\sigma}\lfloor n_i^*\rfloor\right)}\\
\nonumber&&\times\left(\sum_{n\in\mathbb{Z}}(2\alpha_n+\kappa_{n}+\kappa'_{n})\right)e^{-\frac18\delta\sum_{n\in\mathbb{Z}}\left(2\alpha_n+\kappa_n+\kappa'_n\right)}\\
&\leq & \nonumber\frac{480}{e\delta}\sum_{\{n_i\}_{i\geq3}}
\prod_{m\in\mathbb{Z}\atop |m|\leq n_3^*}\left(1+l_m^2\right)e^{-\frac14\delta_1\left(\sum_{i\geq 3}\ln^{\sigma}\lfloor n_i^*\rfloor\right)}\quad \mbox{(by (\ref{042805*}))}\\
\nonumber&\leq&\frac{480}{e\delta}\sup_{\{n_i\}_{i\geq3}}\left(\prod_{m\in\mathbb{Z}\atop|m|\leq n_3^*}\left(1+l_m^2\right)
e^{-\frac18\delta_1\sum_{i\geq3}\ln^{\sigma}\lfloor n_i^*\rfloor}\right)
\\&&\times\sum_{\{n_i\}_{i\geq3}}
e^{-\frac18\delta_1\sum_{i\geq3}\ln^{\sigma}\lfloor n_i^*\rfloor}\label{031910}.
\end{eqnarray}
By (\ref{042807}), one has
\begin{equation}\label{031911}
\sup_{\{n_i\}_{i\geq3}}\left(\prod_{m\in\mathbb{Z}\atop|m|\leq n_3^*}\left(1+l_m^2\right)
e^{-\frac18\delta_1\sum_{i\geq3}\ln^{\sigma}\lfloor n_i^*\rfloor}\right)\leq \exp\left\{{6\left(\frac{32}{\delta_1}\right)^{\frac{1}{\sigma-1}}\cdot \exp{\left(\frac{16}{\delta_1}\right)^{\frac1\sigma}}}
\right\}.
\end{equation}
In view of  (\ref{041809}) and (\ref{122401}), we have
\begin{equation}\label{031912}
\sum_{\{n_i\}_{i\geq3}}
e^{-\frac18\delta_1\sum_{i\geq3}\ln^{\sigma}\lfloor n_i^*\rfloor}\leq \exp\left\{\frac{144}{\delta_1}\cdot\exp\left\{\left(\frac{32}{\delta_1}\right)^{\frac1{\sigma-1}}\right\}\right\}.
\end{equation}
By (\ref{031910}), (\ref{031911}) and (\ref{031912}), we finish the proof of (\ref{041802}).

$\textbf{Case. 2.}\ \nu_1^*>N_1^*.$

In view of (\ref{041801}), one has $n_1^*=\nu_1^*$. Hence,  $n_2$ is determined by $n_1$ and $\{n_i\}_{i\geq 3}$. Similar as $\textbf{Case 1.2}$, we have
\begin{eqnarray*}
I
&\leq& \frac{1}{\delta_2}\exp\left\{\frac{1000}{\delta_1}\cdot\exp\left\{\left(\frac{100}{\delta_1}\right)^{\frac1{\sigma-1}}\right\}\right\}.
\end{eqnarray*}
\end{proof}

Next, we will estimate the symplectic transformation $\Phi_F$ induced by the Hamiltonian function $F$. Actually, we have
\begin{lem}(\textbf{Hamiltonian Flow})\label{E1}
Let $\sigma>2,\rho>0$ and
$$0<\delta<\min \left\{\frac14\rho,3-2\sqrt{2}\right\}.$$ Assume further \begin{equation}\label{042801}\frac{2e}{\delta}\exp\left\{\frac{2000}{\delta}\cdot\exp\left\{\left(\frac{200}{\delta}\right)^{\frac1{\sigma-1}}\right\}\right\}\left\|F\right\|_{\rho-\delta} <\frac 12.
\end{equation}
Then for any Hamiltonian function $H$, we get
\begin{equation}
\left\|H\circ\Phi_F\right\|_\rho
\leq\left(1+\frac{4e}{\delta}\exp\left\{\frac{2000}{\delta}\cdot\exp\left\{\left(\frac{200}{\delta}\right)^{\frac1{\sigma-1}}\right\}\right\}||F||_{\rho-\delta}\right)
||H||_{\rho-\delta}.
\end{equation}
\end{lem}\begin{proof}
 Firstly,  we expand $H\circ\Phi_F$ into the Taylor series
 \begin{equation}\label{3.2}
 H\circ\Phi_F=\sum_{n\geq 0}\frac{1}{n!}H^{(n)},
 \end{equation}
where $H^{(n)}=\{H^{(n-1)},F\}$ and $H^{(0)}=H$.

We will estimate $\left\|H^{(n)}\right\|_\rho$ by using Lemma \ref{010} again and again:
\begin{eqnarray}
\nonumber&&\left\|H^{(n)}\right\|_\rho\\
\nonumber&\leq&\left(\exp\left\{\frac{2000}{\delta}\cdot\exp\left\{\left(\frac{200}{\delta}\right)^{\frac1{\sigma-1}}\right\}\right\}\left\|F\right\|_{\rho-\frac{\delta}{2}}\right)\left(\frac{2n}{\delta}\right)\left\|H^{(n-1)}\right\|_{\rho-\frac{\delta}{2n}}\\
\nonumber&\leq&\left(\exp\left\{\frac{2000}{\delta}\cdot\exp\left\{\left(\frac{200}{\delta}\right)^{\frac1{\sigma-1}}\right\}\right\}
\left\|F\right\|_{\rho-\frac{\delta}{2n}-\frac{\delta}2}\right)^2\left(\frac{2n}{\delta}\right)^2\left\|H^{(n-2)}\right\|_{\rho-\frac{2\delta}{2n}}\\
&&\dots\nonumber\\
\nonumber\label{3.3}&\leq&\left(\exp\left\{\frac{2000}{\delta}\cdot\exp\left\{\left(\frac{200}{\delta}\right)^{\frac1{\sigma-1}}\right\}\right\}
\left\|F\right\|_{\rho-\delta}\right)^n\left(\frac{2n}{\delta}\right)^n\left\|H\right\|_{\rho-\frac{\delta}2}\\
\label{3.3}&\leq&\left(\exp\left\{\frac{2000}{\delta}\cdot\exp\left\{\left(\frac{200}{\delta}\right)^{\frac1{\sigma-1}}\right\}\right\}
\left\|F\right\|_{\rho-\delta}\right)^n\left(\frac{2n}{\delta}\right)^n\left\|H\right\|_{\rho-{\delta}}.
\end{eqnarray}
In view of (\ref{3.2}) and (\ref{3.3}), one has
\begin{eqnarray*}
&&\left\|H\circ\Phi_F\right\|_\rho\\
&\leq&\sum_{n\geq 0}\frac{1}{n!}\left(\exp\left\{\frac{2000}{\delta}\cdot\exp\left\{\left(\frac{200}{\delta}\right)^{\frac1{\sigma-1}}\right\}\right\}
\left\|F\right\|_{\rho-\delta}\right)^{n}\left(\frac{2n}{\delta}\right)^n\left\|H\right\|_{\rho-\delta}\\
&=&\sum_{n\geq 0}\frac{n^n}{n!}\left(\frac{2}{\delta}\exp\left\{\frac{2000}{\delta}\cdot\exp\left\{\left(\frac{200}{\delta}\right)^{\frac1{\sigma-1}}\right\}\right\}
\left\|F\right\|_{\rho-\delta}\right)^{n}\left\|H\right\|_{\rho-\delta}\\
&\leq&\sum_{n\geq 0}\left(\frac{2e}{\delta}\exp\left\{\frac{2000}{\delta}\cdot\exp\left\{\left(\frac{200}{\delta}\right)^{\frac1{\sigma-1}}\right\}\right\}
||F||_{\rho-\delta}\right)^{n}||H||_{\rho-\delta}\\
&&(\mbox{in view of $n^n<n!e^n$)}\\
&\leq & \left(1+\frac{4e}{\delta}\exp\left\{\frac{2000}{\delta}\cdot\exp\left\{\left(\frac{200}{\delta}\right)^{\frac1{\sigma-1}}\right\}\right\}||F||_{\rho-\delta}\right)
||H||_{\rho-\delta}.
\end{eqnarray*}
\end{proof}
\begin{lem}(\textbf{Hamiltonian Vector Field})\label{063004}
Given a Hamiltonian
\begin{equation}\label{050902}
R(q,\bar q)=\sum_{a,k,k'\in\mathbb{N}^{\mathbb{Z}}}B_{akk'}\mathcal{M}_{akk'},
\end{equation}
then for any $\sigma>2$ and $\rho\in(0,3-2\sqrt{2})$, one has
\begin{equation}\label{050907}
\sup_{||q||_{\sigma,\infty}<1}||X_R||_{\sigma,\infty}\leq \exp\left\{\frac{100}{\rho}\cdot\exp\left\{\left(\frac{10}{\rho}\right)^{\frac{1}{\sigma-1}}\right\}\right\}||H||_{\rho}.
\end{equation}\color{black}
\end{lem}
\begin{proof}
Fixed any $j\in\mathbb{Z}$, it suffices to estimate the upper bound for
\begin{equation*}
\left|\frac{\partial{R}}{\partial q_j}e^{\ln^{\sigma}\lfloor j\rfloor}\right|.
\end{equation*}In view of (\ref{050902}),
one has
\begin{equation*}
\frac{\partial{R}}{\partial q_j}=\sum_{a,k,k'\in\mathbb{N}^{\mathbb{Z}}}B_{akk'}\left(\prod_{n\neq j}I_n(0)^{a_n}q_n^{k_n}\bar{q}_n^{k_n'}\right)\left(k_jI_j(0)^{a_j}q_j^{k_j-1}\bar{q}_j^{k_j'}\right).
\end{equation*}
Based on (\ref{042602}), one has
\begin{equation}\label{050903}
\left|B_{akk'}\right|\leq \left\|R\right\|_{\rho}e^{\rho\left(\sum_{n\in\mathbb{Z}}\left(2a_n+k_n+k_n'\right)\ln^{\sigma}\lfloor n\rfloor-2\ln^{\sigma}\lfloor n_1^*\rfloor\right)}.
\end{equation}
Therefore, in view of $||q||_{\sigma,\infty}<1$, (\ref{M13}) and (\ref{050903}), one has
\begin{eqnarray}
&&\left|\frac{\partial{R}}{\partial q_j}e^{\ln^{\sigma}\lfloor j\rfloor}\right|\nonumber\\
\nonumber
&\leq&\nonumber||R||_{\rho}\left|\sum_{a,k,k'\in\mathbb{N}^{\mathbb{Z}}}k_je^{\rho\left(\sum_{n\in\mathbb{Z}}\left(2a_n+k_n+k_n'\right)\ln^{\sigma}\lfloor n\rfloor-2\ln^{\sigma}\lfloor n_1^*\rfloor\right)}e^{-\left(\sum_{n\in\mathbb{Z}}\left(2a_n+k_n+k_n'\right)\ln^{\sigma}\lfloor n\rfloor-2\ln^{\sigma}\lfloor j\rfloor\right)}\right|.
\end{eqnarray}
Then the estimate (\ref{050907}) will follow from
\begin{eqnarray}
\nonumber&&\left|\sum_{a,k,k'\in\mathbb{N}^{\mathbb{Z}}}k_je^{(\rho-1)\left(\sum_{n\in\mathbb{Z}}\left(2a_n+k_n+k_n'\right)\ln^{\sigma}\lfloor n\rfloor\right)-2\rho\ln^{\sigma}\lfloor n_1^*\rfloor+2\ln^{\sigma}\lfloor j\rfloor}\right|\\
&\leq&\exp\left\{\frac{100}{\rho}\cdot\exp\left\{\left(\frac{10}{\rho}\right)^{\frac{1}{\sigma-1}}\right\}\right\}\label{022302}.
\end{eqnarray}
Now we will estimate (\ref{022302}) in the following two cases:

\textbf{Case 1.} $\lfloor j\rfloor \leq \lfloor n_3^*\rfloor$.

Then one has
\begin{eqnarray*}
\left|\sum_{a,k,k'\in\mathbb{N}^{\mathbb{Z}}}k_je^{(\rho-1)\left(\sum_{n\in\mathbb{Z}}\left(2a_n+k_n+k_n'\right)\ln^{\sigma}\lfloor n\rfloor\right)-2\rho\ln^{\sigma}\lfloor n_1^*\rfloor+2\ln^{\sigma}\lfloor j\rfloor}\right|
\leq  \sum_{a,k,k'\in\mathbb{N}^{\mathbb{Z}}}k_je^{\frac{\rho-1}3\sum_{i\geq 1}\ln^{\sigma}\lfloor n_i^*\rfloor}.
\end{eqnarray*}
In view of (\ref{042604}), we have
\begin{eqnarray}
&&\nonumber \sum_{a,k,k'\in\mathbb{N}^{\mathbb{Z}}}k_je^{\frac{\rho-1}3\sum_{i\geq 1}\ln^{\sigma}\lfloor n_i^*\rfloor}\\
\nonumber&\leq&\sum_{a,k,k'\in\mathbb{N}^{\mathbb{Z}}}\left(\sum_{i\geq 1}\ln^{\sigma}\lfloor n_i^*\rfloor\right)e^{\frac{\rho-1}3\sum_{i\geq 1}\ln^{\sigma}\lfloor n_i^*\rfloor}\\
\nonumber&\leq&\frac{12}{e(1-\rho)}\sum_{a,k,k'\in\mathbb{N}^{\mathbb{Z}}}e^{\frac{\rho-1}4\sum_{i\geq 1}\ln^{\sigma}\lfloor n_i^*\rfloor
}\quad\mbox{(in view of (\ref{042805*}))}\\
\nonumber&\leq&12\prod_{n\in\mathbb{Z}}\left({1-e^{-\rho\ln^{\sigma}\lfloor n\rfloor}}\right)^{-1}\prod_{n\in\mathbb{Z}}
\left({1-e^{-\rho\ln^{\sigma}\lfloor n\rfloor}}\right)^{-2}\\
\nonumber&&  \ {\mbox{(in view of $\rho\in(0,3-2\sqrt{2})$ and (\ref{041809}))}}    \\
&\leq&\label{050904}12\exp\left\{\frac{54}{\rho}\cdot\exp\left\{\left(\frac{4}{\rho}\right)^{\frac{1}{\sigma-1}}\right\}\right\},
\end{eqnarray}
where the last inequality is based on (\ref{122401}).

\textbf{Case 2.} $\lfloor j\rfloor > \lfloor n_3^*\rfloor$.

If $k_j\geq 3$, then one has $\lfloor j\rfloor \leq \lfloor n_3^*\rfloor$ which is in \textbf{Case\ 1.}. Hence we always assume $k_j\leq 2$. Then using (\ref{001}), we have
\begin{eqnarray}
\nonumber&&\left|\sum_{a,k,k'\in\mathbb{N}^{\mathbb{Z}}}k_je^{(\rho-1)\left(\sum_{n\in\mathbb{Z}}\left(2a_n+k_n+k_n'\right)\ln^{\sigma}\lfloor n\rfloor\right)-2\rho\ln^{\sigma}\lfloor n_1^*\rfloor+2\ln^{\sigma}\lfloor j\rfloor}\right|\\
\label{201606032}&\leq& 2\left|\sum_{a,k,k'\in\mathbb{N}^{\mathbb{Z}}}e^{\frac{\rho-1}2\sum_{i\geq 3}\ln^{\sigma}\lfloor n_i^*\rfloor}\right|.
\end{eqnarray}

Denote $(n_i)_{i\geq1}$ by the system $\left\{n:\mbox{where $n$ is repeated $2a_n+k_n+k_n'$ times}\right\}$.  Without loss of generality, we assume that $n_1=n_1^*$ and $n_2=n_2^*$. Note that $\{n_i\}_{i\geq 1}$ is given, then $\{2a_n+k_n+k'_n\}_{n\in\mathbb{Z}}$ is specified, and hence $(a,k,k')$ is specified up to a factor of
\begin{equation}\label{030204}\prod_{n\in\mathbb{Z}}\left(1+l_n^2\right),
\end{equation}
where
$$l_n=\#\{j:n_j=n\}.$$
In view of $\lfloor j\rfloor >\lfloor n_3^*\rfloor$ again, one has $j\in\left\{n_1,n_2\right\}$ and
\begin{equation*}
l_{n_1}+l_{n_2}\leq 2,
\end{equation*}
which implies
\begin{equation}\label{022301}
\prod_{n\in\mathbb{Z}\atop n=n_1,n_2}\left(1+l_n^2\right)\leq 5.
\end{equation}Furthermore if $(n_i)_{i\geq 3}$ and $j$ are given, $n_1$ and $n_2$ are uniquely determined. Then using (\ref{030204}) and (\ref{022301}) one has
\begin{eqnarray*}\label{050905}
\nonumber(\ref{201606032})
&\leq&\nonumber 10\left|\sum_{\left(n_i\right)_{i\geq3}}\prod_{n\in\mathbb{Z}\atop|n|\leq n_3^*}\left(1+l_n^2\right)e^{\frac{\rho-1}2\sum_{i\geq 3}\ln^{\sigma}\lfloor n_i^*\rfloor}\right| \\
\nonumber&\leq&10\left|\left(\sum_{(n_i)_{i\geq3}}e^{-\rho\sum_{i\geq 3}\ln^{\sigma}\lfloor n_i^*\rfloor}\right)\sup_{(n_i)_{i\geq3}}\left(\prod_{|n|\leq n_3^*}\left(1+l_n^2\right)e^{-\rho\sum_{i\geq 3}\ln^{\sigma}\lfloor n_i^*\rfloor}\right)\right|\\
\nonumber
\nonumber
\nonumber
&\leq&\label{022303}10\exp\left\{\frac{18}{\rho}\cdot\exp\left\{\left(\frac{4}{\rho}\right)^{\frac{1}{\sigma-1}}\right\}\right\}
\cdot\exp\left\{6\left(\frac{2}{\rho}\right)^{\frac{1}{\sigma-1}}\cdot\exp\left\{\left(\frac{1}{\rho}\right)^{\frac1{\sigma}}\right\}\right\} ,
\end{eqnarray*}
where the last inequality is based on (\ref{041809}), (\ref{122401}) and (\ref{042807}).

In view of the above two cases, we finish the proof of (\ref{022302}).
\end{proof}
\section{KAM iteration}
\subsection{Derivation of homological equations}
The proof of Main Theorem employs the rapidly
converging iteration scheme of Newton type to deal with small divisor problems
introduced by Kolmogorov, involving the infinite sequence of coordinate transformations.
At the $s$-th step of the scheme, a Hamiltonian
$H_{s} = N_{s} + R_{s}$
is considered, as a small perturbation of some normal form $N_{s}$. A transformation $\Phi_{s}$ is
set up so that
$$H_{s}\circ \Phi_{s} = N_{s+1} + R_{s+1}$$
with another normal form $N_{s+1}$ and a much smaller perturbation $R_{s+1}$. We drop the index $s$ of $H_{s}, N_{s}, R_{s}, \Phi_{s}$ and shorten the index $s+1$ as $+$.

Now consider the Hamiltonian $H$ of the form
\begin{eqnarray}\label{N1}
{H}=N+R,
\end{eqnarray}
where
\begin{equation*}
N=\sum_{n\in\mathbb{Z}}\left(n^2+\widetilde V_n\right)|q_n|^2,
\end{equation*}
and
\begin{equation*}
R=R_0+R_1+R_2
\end{equation*}
with $\left|\widetilde V_n\right|\leq 2$ for all $n\in\mathbb{Z}$,
\begin{eqnarray*}
{R}_0&=&\sum_{a,k,k'\in\mathbb{N}^{\mathbb{Z}}\atop\mbox{supp}\ k\bigcap \mbox{supp}\ k'=\emptyset}B_{akk'}\mathcal{M}_{akk'},\\
{R}_1&=&\sum_{n\in\mathbb{Z}}J_n\left(\sum_{a,k,k'\in\mathbb{N}^{\mathbb{Z}}\atop\mbox{supp}\ k\bigcap \mbox{supp}\ k'=\emptyset}B_{akk'}^{(n)}\mathcal{M}_{akk'}\right),\\
{R}_2&=&\sum_{n_1,n_2\in\mathbb{Z}}J_{n_1}J_{n_2}\left(\sum_{a,k,k'\in\mathbb{N}^{\mathbb{Z}}\atop\mbox{no assumption}}B_{akk'}^{(n_1,n_2)}\mathcal{M}_{akk'}\right).
\end{eqnarray*}
We desire to eliminate the terms $R_0,R_1$ in (\ref{N1}) by the coordinate transformation $\Phi$, which is obtained as the time-1 map $X_F^{t}|_{t=1}$ of a Hamiltonian
vector field $X_F$ with $F=F_0+F_1$. Let ${F}_{0}$ (resp. ${F}_{1}$) has the form of ${R}_0$ (resp. ${R}_{1}$),
that is \begin{eqnarray}
&&{F}_0=\sum_{a,k,k'\in\mathbb{N}^{\mathbb{Z}}\atop\mbox{supp}\ k\bigcap \mbox{supp}\ k'=\emptyset}F_{akk'}\mathcal{M}_{akk'},\\
&&{F}_1=\sum_{n\in\mathbb{Z}}J_n\left(\sum_{a,k,k'\in\mathbb{N}^{\mathbb{Z}}\atop\mbox{supp}\ k\bigcap \mbox{supp}\ k'=\emptyset}F_{akk'}^{(n)}\mathcal{M}_{akk'}\right),
\end{eqnarray}
and the homological equations become
\begin{equation}\label{4.27}
\{N,{F}\}+R_0+R_{1}=[R_0]+[R_1],
\end{equation}
where
\begin{equation}\label{051501}
[R_0]=\sum_{a\in\mathbb{N}^{\mathbb{Z}}}B_{a00}\mathcal{M}_{a00},
\end{equation}
and
\begin{equation}\label{051502}
[R_1]=\sum_{n\in\mathbb{Z}}J_n\sum_{a\in\mathbb{N}^{\mathbb{Z}}}B_{a00}^{(n)}\mathcal{M}_{a00}.
\end{equation}
The solutions of the homological equations (\ref{4.27}) are given by
\begin{equation}\label{051304}
F_{akk'}=\frac{B_{akk'}}{\sum_{n\in\mathbb{Z}}\left(k_n-k^{'}_n\right)(n^2+\widetilde{V}_n)},
\end{equation}
where
\begin{equation}\label{051305}
F_{akk'}^{(m)}=\frac{B_{akk'}^{(m)}}{\sum_{n\in\mathbb{Z}}\left(k_n-k^{'}_n\right)(n^2+\widetilde{V}_n)},
\end{equation}
and the new Hamiltonian ${H}_{+}$ has the form
\begin{eqnarray}
H_{+}\nonumber&=&H\circ\Phi\\
&=&\nonumber N+\{N,F\}+R_0+R_1\\
&&\nonumber+\int_{0}^1\{(1-t)\{N,F\}+R_0+R_1,F\}\circ X_F^{t}\ \mathrm{d}{t}
+\nonumber R_2\circ X_F^1\\
&=&\label{051401}N_++R_+,
\end{eqnarray}
where
\begin{equation}\label{051402}
N_+=N+[R_0]+[R_1],
\end{equation}
and
\begin{equation}\label{051403}
R_+=\int_{0}^1\{(1-t)\{N,F\}+R_0+R_1,F\}\circ X_F^{t}\ \mathrm{d} t+R_2\circ X_F^1.
\end{equation}
\subsection{The new norm}To estimate
the solutions of the homological equations, we will define a new norm for the Hamiltonian ${R}$ of the form as in (\ref{N1}) as follows:
\begin{equation}
||{R}||_{\rho}^{+}=\max\{||R_0||_\rho^{+},||R_1||_\rho^{+}|,||R_2||_\rho^{+}\},
\end{equation}
where
\begin{eqnarray}
&&\label{051302}\left\|R_0\right\|_\rho^{+}=\sup_{a,k,k'\in\mathbb{N}^{\mathbb{Z}}}\frac{\left|B_{akk'}\right|}
{e^{\rho\left(\sum_{n\in\mathbb{Z}}\left(2a_n+k_n+k_n'\right)\ln^{\sigma}\lfloor n\rfloor-2\ln^{\sigma}\lfloor n_1^*\rfloor\right)}},\\
&&\left\|R_1\right\|_\rho^{+}=\sup_{a,k,k'\in\mathbb{N}^{\mathbb{Z}}\atop m\in\mathbb{Z}}\frac{\left|B^{(m)}_{akk'}\right|}{e^{\rho\left(\sum_{n\in\mathbb{Z}}
\left(2a_n+k_n+k_n'\right)\ln^{\sigma}\lfloor n\rfloor+2\ln^{\sigma}\lfloor m\rfloor-2\ln^{\sigma}\lfloor n_1^*\rfloor\right)}}\label{031920},
\end{eqnarray}
and
\begin{equation}\label{031921}
\left\|R_2\right\|_\rho^{+}=\sup_{a,k,k'\in\mathbb{N}^{\mathbb{Z}}\atop
m_1,m_2\in\mathbb{Z}}\frac{\left|B^{(m_1,m_2)}_{akk'}\right|}
{e^{\rho\left(\sum_{n\in\mathbb{Z}}\left(2a_n+k_n+k_n'\right)\ln^{\sigma}\lfloor n\rfloor
+2\ln^{\sigma}\lfloor m_1\rfloor+2\ln^{\sigma}\lfloor m_2\rfloor-2\ln^{\sigma}\lfloor n_1^*\rfloor\right)}}.
\end{equation}
\begin{rem}
Here we abuse the notation $n_1^*$. In fact, in (\ref{051302})
\begin{equation*}
n_1^*=\max\left\{|n|:a_n+k_n+k_n'\neq 0\right\};
\end{equation*}
in (\ref{031920})
\begin{equation*}
n_1^*=\max\left\{\max\left\{|n|:a_n+k_n+k_n'\neq 0\right\},|m|\right\};
\end{equation*}
and in (\ref{031921})
\begin{equation*}
n_1^*=\max\left\{\max\left\{|n|:a_n+k_n+k_n'\neq 0\right\},|m_1|,|m_2|\right\}.
\end{equation*}
\end{rem}
Then one has the following estimates:
\begin{lem}\label{051301}
Given any $\rho,\delta>0$ and a Hamiltonian $R$, one has
\begin{equation}\label{N6}
||R||_{\rho+\delta}^{+}\leq\exp\left\{3\left(\frac6{\delta}\right)^{\frac1{\sigma-1}}\cdot \exp\left\{\left(\frac6{\delta}\right)^{\frac1{\sigma}}\right\} \right\}||R||_{\rho}
\end{equation}
and
\begin{equation}\label{N7}
||R||_{\rho+\delta}\leq\frac{64}{e^2\delta^2}||R||_{\rho}^{+}.
\end{equation}
\end{lem}
\begin{proof}
Firstly, we will prove the inequality (\ref{N6}).
 Fixed $a,k,k'\in\mathbb{N}^{\mathbb{Z}}$, consider the term
$$\mathcal{M}_{akk'}=\prod_{n\in\mathbb{Z}}I_n(0)^{a_n}q_n^{k_n}\bar q_n^{k_n'}$$ satisfying $k_nk_n'=0$ for all $n\in\mathbb{Z}$. It is easy to see that $\mathcal{M}_{akk'}$ comes from some parts of the terms $\mathcal{M}_{\alpha\kappa\kappa'}$ with no assumption for $\kappa$ and $\kappa'$.
Write $\mathcal{M}_{\alpha\kappa\kappa'}$ in the form of
\begin{equation*}
\mathcal{M}_{\alpha\kappa\kappa'}=\mathcal{M}_{\alpha bll'}=\prod_{n\in\mathbb{Z}}I_n(0)^{\alpha_n}I_n^{b_n}q_n^{l_n}{\bar q_n}^{l_n'}
\end{equation*}
where
\begin{equation*}
b_n=\kappa_n\wedge \kappa_n',\quad l_n=\kappa_n-b_n,\quad l_n'=\kappa_n'-b_n
\end{equation*}
and then
$l_nl_n'=0$ for all $n\in\mathbb{Z}$.

Express the term
\begin{equation*}
\prod_{n\in\mathbb{Z}}I_n^{b_n}=\prod_{n\in\mathbb{Z}}\left(I_n(0)+J_n\right)^{b_n}
\end{equation*}by the monomials of the following form
\begin{equation*}
\prod_{n\in\mathbb{Z}}I_n(0)^{b_n},
\end{equation*}
\begin{equation*}
b_mI_m(0)^{b_m-1}J_m\left(\prod_{n\in\mathbb{Z}\atop n\neq m}I_n(0)^{b_n}\right),\qquad m\in\mathbb{Z},
\end{equation*}
\begin{equation*}
\left(\prod_{n\in\mathbb{Z}\atop n<m}I_n(0)^{b_n}\right)\left(b_m(b_m-1)I_m(0)^{r}J_m^2I_m^{b_m-r-2}\right)\left(\prod_{n\in\mathbb{Z}\atop n>m}I_n^{b_n}\right),\qquad m\in\mathbb{Z},\ 0\leq r\leq b_m-2.
\end{equation*}
and
\begin{eqnarray*}
&&\left(\prod_{n<m_1}I_n(0)^{b_n}\right)\left(b_{m_1}I_{m_1}(0)^{b_{m_1}-1}J_{m_1}\right)\left(\prod_{m_1<n<m_2}I_n(0)^{b_n}\right)
\\
&&\nonumber\times\left(b_{m_2}I_{m_2}(0)^{r}J_{m_2}I_{m_2}^{b_{m_2}-r-1}\right)\left(\prod_{n>m_2}I_n^{b_n}\right)\qquad m_1<m_2,\ 0\leq r\leq b_{m_2}-1.
\end{eqnarray*}
Now we will estimate the bounds for the coefficients respectively. For any given $n\in\mathbb{Z}$,
\begin{equation}\label{030510}
I_n(0)^{a_n}q_n^{k_n}\bar q_n^{k_n'}=\sum_{b_n=\kappa_n\wedge \kappa_n'}I_n(0)^{\alpha_n+b_n}q_n^{\kappa_n-b_n}\bar q_n^{\kappa_n'-b_n}.
\end{equation}
Hence, one has
\begin{equation}\label{N2}
a_n=\alpha_n+b_n,
\end{equation}
and
\begin{equation}\label{N3}
k_n=\kappa_n-b_n,\qquad k_n'=\kappa_n'-b_n.
\end{equation}
Therefore, if $0\leq\alpha_n\leq a_n$ is chosen, so $b_n,\kappa_n,\kappa_n'$ are determined.

On the other hand, by (\ref{042602}) we have
\begin{eqnarray}
\nonumber\left|B_{\alpha\kappa\kappa'}\right|
&\leq&\nonumber \left\|R\right\|_{\rho}e^{\rho\left(\sum_{n\in\mathbb{Z}}\left(2\alpha_n+\kappa_n+\kappa_n'\right)\ln^{\sigma}\lfloor n\rfloor-2\ln^{\sigma}\lfloor n_1^*\rfloor\right)}\qquad\qquad\qquad\qquad\ \\
&=&\nonumber\left\|R\right\|_{\rho}e^{\rho\left(\sum_{n\in\mathbb{Z}}\left(2\alpha_n+\left(k_n+a_n-\alpha_n\right)+\left(k_n'+a_n-\alpha_n\right)\right)\ln^{\sigma}\lfloor n\rfloor-2\ln^{\sigma}\lfloor n_1^*\rfloor\right)}\quad\\
&&\nonumber{(\mbox{in view of (\ref{N2}) and (\ref{N3})})}\\
&=&\nonumber\left\|R\right\|_{\rho}e^{\rho\left(\sum_{n\in\mathbb{Z}}\left(2a_n+k_n+k_n'\right)\ln^{\sigma}
\lfloor n\rfloor-2\ln^{\sigma}\lfloor n_1^*\rfloor\right)}.
\end{eqnarray}
Using (\ref{030510}), one has
\begin{equation}\label{N4}
\left|B_{akk'}\right|\leq\left\|R\right\|_{\rho}\prod_{n\in\mathbb{Z}}\left(1+a_n\right)
e^{\rho\left(\sum_{n\in\mathbb{Z}}\left(2a_n+k_n+k_n'\right)\ln^{\sigma}\lfloor n\rfloor-2\ln^{\sigma}\lfloor n_1^*\rfloor\right)}.
\end{equation}
Similarly,
\begin{eqnarray*}
\left|B_{akk'}^{(m)}\right|&\leq& \left\|R\right\|_{\rho}\left(\prod_{n\neq m}\left(1+a_n\right)\right)\left(1+a_m\right)^2e^{\rho\left(\sum_{n\in\mathbb{Z}}
\left(2a_n+k_n+k_n'\right)\ln^{\sigma}\lfloor n\rfloor+2\ln^{\sigma}\lfloor m\rfloor-2\ln^{\sigma}\lfloor n_1^*\rfloor\right)},\\
\left|B_{akk'}^{(m,m)}\right|&\leq& \left\|R\right\|_{\rho}\left(\prod_{n\neq m}\left(1+a_n\right)\right)\left(1+a_m\right)^3e^{\rho\left(\sum_{n\in\mathbb{Z}}
\left(2a_n+k_n+k_n'\right)\ln^{\sigma}\lfloor n\rfloor+4\ln^{\sigma}\lfloor m\rfloor-2\ln^{\sigma}\lfloor n_1^*\rfloor\right)},\\
\left|B_{akk'}^{(m_1,m_2)}\right|&\leq& \left\|R\right\|_{\rho}\left(\prod_{n<m_1}\left(1+a_n\right)\right)\left(1+a_{m_1}\right)^2
\left(\prod_{m_1<n<m_2 }\left(1+a_n\right)\right)\\&&\times \left(1+a_{m_2}\right)^2e^{\rho\left(\sum_{n\in\mathbb{Z}}\left(2a_n+k_n+k_n'\right)\ln^{\sigma}\lfloor n\rfloor+2\ln^{\sigma}\lfloor m_1\rfloor+2\ln^{\sigma}\lfloor m_2\rfloor-2\ln^{\sigma}\lfloor n_1^*\rfloor\right)}.
\end{eqnarray*}
In view of (\ref{051302}) and (\ref{N4}), we have
\begin{equation}\label{N5}
\left\|R_0\right\|_{\rho+\delta}^{+}
\leq\left\|R\right\|_{\rho}\prod_{n\in\mathbb{Z}}\left(1+a_n\right)e^{-\delta\left(\sum_{n\in\mathbb{Z}}\left(
2a_n+k_n+k_n'\right)\ln^{\sigma}\lfloor n\rfloor-2\ln^{\sigma}\lfloor n_1\rfloor\right)}.
\end{equation}
Now we will estimate the upper bound of
\begin{equation}
\label{051303}\prod_{n\in\mathbb{Z}}\left(1+a_n\right)e^{-\delta\left(\sum_{n\in\mathbb{Z}}
\left(2a_n+k_n+k_n'\right)\ln^{\sigma}\lfloor n\rfloor-2\ln^{\sigma}\lfloor n_1^*\rfloor\right)}.
\end{equation}

\textbf{Case 1.} $n_1^*=n_3^*.$

Using (\ref{001}), one has
\begin{eqnarray}
(\ref{051303})
&\leq&\nonumber \prod_{n\in\mathbb{Z}}\left(1+a_n\right)e^{-\frac12{\delta}\sum_{i\geq3}\ln^{\sigma}\lfloor n_i\rfloor}\\
&\leq&\nonumber\prod_{n\in\mathbb{Z}}\left(1+a_n\right)^{-\frac16{\delta}\sum_{i\geq1}\ln^{\sigma}\lfloor n_i\rfloor}\\
&=&\nonumber\prod_{n\in\mathbb{Z}}\left(1+a_n\right)e^{-\frac16\delta\sum_{n\in\mathbb{Z}}
\left(2a_n+k_n+k_n'\right)\ln^{\sigma}\lfloor n\rfloor}\\
&\leq&\nonumber\prod_{n\in\mathbb{Z}}\left((1+a_n)e^{-\frac13{\delta}a_n\ln^{\sigma}\lfloor n\rfloor}\right)\\
&\leq&\exp\left\{3\left(\frac6{\delta}\right)^{\frac1{\sigma-1}}\cdot \exp\left\{\left(\frac6{\delta}\right)^{\frac1{\sigma}}\right\} \right\}\label{030205},
\end{eqnarray}
where the last inequality is based on (\ref{042807}).

\textbf{Case 2.} $n_1^*>n_2^*=n_3^*.$

In this case, one has $a_{n_1^*}=a_{-n_1^*}=0$, otherwise $n_1^*=n_2^*$.
Then we have
\begin{eqnarray}
\nonumber(\ref{051303})
&=&\left(\prod_{|n|\leq n_2^*}\left(1+a_n\right)e^{-\frac12\delta\sum_{i\geq3}\ln^{\sigma}\lfloor n_i^*\rfloor}\right)
\\
&\leq&\nonumber\prod_{|n|\leq n_2^*}(1+a_n)e^{-\frac14\delta\sum_{i\geq2}\ln^{\sigma}\lfloor n_i^*\rfloor}\\
&\leq&\label{030206} \exp\left\{3\left(\frac4{\delta}\right)^{\frac1{\sigma-1}}\cdot \exp\left\{\left(\frac4{\delta}\right)^{\frac1{\sigma}}\right\} \right\},
\end{eqnarray}
where the last inequality follows from the proof of (\ref{030205}).

\textbf{Case 3.} $n_2^*>n_3^*.$

In this case, one has $a_{n_1^*}+a_{-n_1^*}+a_{n_2^*}+a_{-n_2^*}\leq1$ which implies
\begin{equation*}
\prod_{|n|=n_1^*,n_2^*}(1+a_n)\leq 2.
\end{equation*}
 Hence we have
\begin{eqnarray*}
(\ref{051303})
&\leq&2\left(\prod_{|n|\leq n_3^*}(1+a_n)e^{-\frac12\delta\sum_{i\geq3}\ln^{\sigma}\lfloor n_i^*\rfloor}\right)
\\
&\leq&2\exp\left\{3\left(\frac2{\delta}\right)^{\frac1{\sigma-1}}\cdot \exp\left\{\left(\frac2{\delta}\right)^{\frac1{\sigma}}\right\} \right\},
\end{eqnarray*}
where the last inequality follows from the proof of (\ref{030205}).

In view of the above three cases, one has
\begin{equation*} ({\ref{051303}})\leq \exp\left\{3\left(\frac6{\delta}\right)^{\frac1{\sigma-1}}\cdot \exp\left\{\left(\frac6{\delta}\right)^{\frac1{\sigma}}\right\} \right\}.
\end{equation*}
Similarly, we get
\begin{eqnarray*}
&&||R_1||_{\rho+\delta}^{+},||R_2||_{\rho+\delta}^{+}\leq\exp\left\{3\left(\frac6{\delta}\right)^{\frac1{\sigma-1}}\cdot \exp\left\{\left(\frac6{\delta}\right)^{\frac1{\sigma}}\right\} \right\}||R||_{\rho},
\end{eqnarray*}
Then we have
\begin{equation*}
||R||_{\rho+\delta}^{+}\leq\exp\left\{3\left(\frac6{\delta}\right)^{\frac1{\sigma-1}}\cdot \exp\left\{\left(\frac6{\delta}\right)^{\frac1{\sigma}}\right\} \right\}||R||_{\rho}.
\end{equation*}

On the other hand, the coefficient of $\mathcal{M}_{\alpha bll'}$ increases by at most a factor $$\left(\sum_{n\in\mathbb{Z}}(\alpha_n+b_n)\right)^2.$$ Then one has
\begin{eqnarray}
\nonumber\left\|R\right\|_{\rho+\delta}
&\leq&\nonumber\left\|R\right\|_{\rho}^{+}\left(\sum_{n\in\mathbb{Z}}(\alpha_n+b_n)\right)^2e^{-\delta\left(\sum_{n\in\mathbb{Z}}\left(2a_n+k_n+k_n'\right)\ln^{\sigma}\lfloor n\rfloor-2\lfloor n_1^*\rfloor\right)}\\
&\leq&\nonumber\left\|R\right\|_{\rho}^{+}\left(2\sum_{i\geq 3}\ln^{\sigma}\lfloor n_i\rfloor\right)^2 e^{-\frac12\delta\sum_{i\geq3}\ln^{\sigma}\lfloor n_i^*\rfloor}\quad (\mbox{in view of (\ref{042605})})\\
&\leq&\nonumber\frac{64}{e^2\delta^2}||R||_{\rho}^{+},
\end{eqnarray}
where the last inequality is based on Lemma \ref{8.6} with $p=2$.
\end{proof}

\subsection{KAM Iteration}\label{031501}

Now we give the precise
set-up of iteration parameters. For any $\sigma>2$, let $s\geq0$ be the $s$-th KAM
step.
 \begin{itemize}
 \item[]$\delta_{s}=\frac{\rho_0}{(s+4)\ln^2(s+4)},\qquad$ with $\rho_0= \frac{3-2\sqrt{2}}{100}$,

 \item[]$\rho_{s+1}=\rho_{s}+3\delta_s,$

 \item[]$\epsilon_s=\epsilon_{0}^{(\frac{3}{2})^s}$, which dominates the size of
 the perturbation,

 \item[]$\lambda_s=\epsilon_s^{0.01}$,

 \item[]$\eta_{s+1}=\frac{1}{20}\lambda_s\eta_s$, with $\eta_0=\lambda_0$,

 \item[]$d_{s+1}=d_{s}+\frac{1}{\pi^2(s+1)^2}$, with $d_0=0$,

 \item[]$D_s=\{(q_n)_{n\in\mathbb{Z}}:\frac{1}{2}+d_s\leq|q_n|e^{\ln^{\sigma}\lfloor n\rfloor}\leq1-d_s\}$.
 \end{itemize}
\begin{rem}
Then one has
\begin{equation*}
\sum_{s\geq 0}\delta_s\leq \frac53\rho_0,
\end{equation*}
\begin{equation*}
\rho_0\leq \rho_s\leq 6\rho_0<\frac17,\qquad \forall\ s\geq 0,
\end{equation*}
\begin{equation*}\delta_s<\min \left\{\frac14\rho_s,3-2\sqrt{2}\right\},\qquad \forall\ s\geq 0.
\end{equation*}
\end{rem}
Denote the complex cube of size $\lambda>0$:
\begin{equation*}\label{M9}
\mathcal{C}_{\lambda}\left(\widetilde{V}\right)=\left\{(V_n)_{n\in\mathbb{Z}}\in\mathbb{C}^{\mathbb{Z}}:\left|V_n-\widetilde{V}_n\right|\leq \lambda\right\}.
\end{equation*}

\begin{lem}{\label{IL}}
Suppose $H_{s}=N_{s}+R_{s}$ is real analytic on $D_{s}\times\mathcal{C}_{\eta_{s}}\left(V_{s}^*\right)$,
where $$N_{s}=\sum_{n\in\mathbb{Z}}\left(n^2+\widetilde V_{n,s}\right)\left|q_n\right|^2$$ is a normal form with coefficients satisfying
\begin{eqnarray}
\label{198}&&\widetilde{V}_{s}\left(V_{s}^*\right)=\omega,\\
\label{199}&&\left|\left|\frac{\partial \widetilde{V}_s}{{\partial V}}-Id\right|\right|_{l^{\infty}\rightarrow l^{\infty}}<d_s\epsilon_{0}^{\frac{1}{10}},
\end{eqnarray}
and $R_{s}=R_{0,s}+R_{1,s}+R_{2,s}$ satisfying
\begin{eqnarray}
\label{200}&&||R_{0,s}||_{\rho_{s}}^{+}\leq \epsilon_{s},\\
\label{201}&&||R_{1,s}||_{\rho_{s}}^{+}\leq \epsilon_{s}^{0.6},\\
\label{202}&&||R_{2,s}||_{\rho_{s}}^{+}\leq (1+d_s)\epsilon_0.
\end{eqnarray}
Then for all $V\in\mathcal{C}_{\eta_{s}}\left(V_{s}^*\right)$ satisfying $\widetilde V_{s}(V)\in\mathcal{C}_{\lambda_s}(\omega)$, there exist real analytic symplectic coordinate transformations
$\Phi_{s+1}:D_{s+1}\rightarrow D_{s}$ satisfying
\begin{eqnarray}
\label{203}&&\left\|\Phi_{s+1}-id\right\|_{\sigma,\infty}\leq \epsilon_{s}^{0.5},\\
\label{204}&&\left\|D\Phi_{s+1}-Id\right\|_{(\sigma,\infty)\rightarrow(\sigma,\infty)}\leq \epsilon_{s}^{0.5},
\end{eqnarray}
such that for
$H_{s+1}=H_{s}\circ\Phi_{s+1}=N_{s+1}+R_{s+1}$, the same assumptions as above are satisfied with `$s+1$' in place of `$s$', where $\mathcal{C}_{\eta_{s+1}}\left(V_{s+1}^*\right)\subset\widetilde V_{s}^{-1}(\mathcal{C}_{\lambda_s}(\omega))$ and
\begin{equation}\label{206}
\left\|\widetilde{V}_{s+1}-\widetilde{V}_{s}\right\|_{\infty}\leq\epsilon_{s}^{0.5},
\end{equation}
\begin{equation}\label{205}
\left\|V_{s+1}^*-V_{s}^*\right\|_{\infty}\leq2\epsilon_{s}^{0.5}.
\end{equation}
\end{lem}
\begin{proof}
We will estimate the solution of the homological equation firstly (see (\ref{051304}) and (\ref{051305})).
In the step $s\rightarrow s+1$, there is saving of a factor
\begin{equation*}
e^{-\delta_{s}\left(\sum_{n\in\mathbb{Z}}\left(2a_n+k_n+k'_n\right)\ln^{\sigma}\lfloor n\rfloor-2\ln^{\sigma}\lfloor n_1^*\rfloor\right)}\leq e^{-\frac12\delta_{s}\sum_{i\geq3}\ln^{\sigma}\lfloor n_i^*\rfloor}.
\end{equation*}
Recalling after this step, we need
\begin{eqnarray*}
&&\left\|R_{0,s+1}\right\|_{\rho_{s+1}}^{+}\leq \epsilon_{s+1},\\
&&\left\|R_{1,s+1}\right\|_{\rho_{s+1}}^{+}\leq \epsilon_{s+1}^{0.6}.
\end{eqnarray*}
Consequently, in $R_{i,s}\ (i=0,1)$, it suffices to eliminate the nonresonant monomials $\mathcal{M}_{akk'}$ for which
\begin{equation*}
e^{-\frac12\delta_{s}\sum_{i\geq3}\ln^{\sigma}\lfloor n_i^*\rfloor}\geq\epsilon_{s+1},
\end{equation*}
that is
\begin{equation}\label{M1}
\sum_{i\geq3}\ln^{\sigma}\lfloor n_i^*\rfloor\leq\frac{2(s+4)\ln^2(s+4)}{\rho_0}\cdot\ln\frac{1}{\epsilon_{s+1}}.
\end{equation}
On the other hand, in the small divisors analysis, we need only impose Diophantine conditions when (\ref{041803}) holds. By Lemma \ref{a1} and (\ref{M1}), one has
\begin{equation}\label{M2}
\sum_{n\in\mathbb{Z}}|k_n-k_n'|\ln^{\sigma}\lfloor n\rfloor
\leq3\cdot 4^{\sigma}\cdot\frac{2(s+4)\ln^2(s+4)}{\rho_0}\cdot\ln\frac{1}{\epsilon_{s+1}}:= B_s.
\end{equation}
From (\ref{M2}), we need only impose condition on $\left(\widetilde{V}_n\right)_{|n|\leq N_{s}^*}$, where
\begin{equation}\label{032401}
N_{s}^*\sim \exp\left\{{B_s^{\frac1{\sigma}}}\right\}.
\end{equation}
Correspondingly, the Diophantine condition becomes
\begin{equation}\label{M4}
\left|\left|\sum_{|n|\leq N_{s}^*}(k_n-k'_n)\widetilde{V}_{n,s}\right|\right|\geq \gamma\prod_{|n|\leq N_{s}^*}\frac{1}{1+(k_n-k'_n)^2\langle n\rangle^4},
\end{equation}
and we finished the truncation step.

Now we will obtain lower bound on the right hand side of (\ref{M4}). More concretely, let
\begin{equation}\label{030701}N_s=\frac{B_s^{\frac{\sigma-1}{\sigma}}}{\ln^{\sigma} B_s}
\end{equation} and we have
\begin{eqnarray}
&&\nonumber\prod_{|n|\leq N_{s}^*}\frac{1}{1+(k_n-k'_n)^2\langle n\rangle^4}\\
\nonumber&=&\exp\left\{\sum_{|n|\leq N_s}\ln\left(\frac{1}{1+(k_n-k'_n)^2\langle n\rangle^4}\right)+\sum_{N_s<|n|\leq N_s^*}\ln\left(\frac{1}{1+(k_n-k'_n)^2\langle n\rangle^4}\right)\right\}\\
\nonumber
&\geq& \exp\left\{-10\sum_{|n|\leq N_s,k_n\neq k'_n}\left(\ln\left(|k_n-k'_n|+8\right)\right)\ln\lfloor n\rfloor-10\sum\limits_{|n|> N_s,k_n\neq k'_n}|k_n-k'_n|\ln\lfloor n\rfloor\right\}\\
\nonumber
\nonumber&\geq& \exp\left\{-60N_s B_s^{\frac1{\sigma}}-10B_s(\ln N_s)^{^{1-\sigma}}\right\}\qquad\qquad \ \mbox{(in view of (\ref{M2}))}\\
\label {M6}&\geq& \exp\left\{-{100B_s}{\left(\ln B_s\right)^{1-\sigma }}\right\},
\end{eqnarray}where the last inequality is based on (\ref{030701}).
In view of (\ref{M2}), one has
\begin{equation*}
\ln B_s=\ln\left(\frac{6\cdot4^{\sigma}}{\rho_0}\right)+\ln (s+4)+2\ln\ln(s+4)+(s+1)\ln\frac32+\ln\left|\ln\epsilon_0\right|
\end{equation*}
and then
\begin{eqnarray}
&&\nonumber100B_s(\ln B_s)^{1-\sigma}\\
&=&\nonumber\left(300\cdot 4^{\sigma}\cdot\frac{2(s+4)\ln^2(s+4)}{\rho_0}\cdot\ln\frac{1}{\epsilon_{s+1}}\right)\\
&&\nonumber\left(\ln\left(\frac{6\cdot4^{\sigma}}{\rho_0}\right)+\ln (s+4)+2\ln\ln(s+4)+(s+1)\ln\frac32+\ln\left|\ln\epsilon_0\right|\right)^{1-\sigma}\\
&\leq&\nonumber \ln \frac{1}{\epsilon_s}\cdot\left(900\cdot 4^{\sigma}\cdot\frac{2(s+4)\ln^2(s+4)}{\rho_0}\right)\left((s+1)\ln\frac32+\ln\left|\ln\epsilon_0\right|\right)^{1-\sigma}\\
&\leq&0.01\cdot \ln \frac{1}{\epsilon_s}\label{022402},
\end{eqnarray}
where using $\sigma>2$ and $\epsilon_0$ is sufficiently small depending on $\sigma$ only. Furthermore by (\ref{M6}) and (\ref{022402}), we have
\begin{eqnarray}
\prod_{|n|\leq N_s^{*}}\frac{1}{1+(k_n-k'_n)^2\langle n\rangle^4}
 &\geq&\exp\left\{-0.01\cdot \ln \frac{1}{\epsilon_s}\right\}=\lambda_s
\label{M7}.
\end{eqnarray}
Assuming $V\in \mathcal{C}_{\lambda_s}(\widetilde{V}_s)$, from the lower bound (\ref{M7}), the relation (\ref{M4}) remains true if we substitute $V$ for $\widetilde{V}_s$. Moreover, there is analyticity on $\mathcal{C}_{\lambda_s}(\widetilde{V}_s)$. The transformations $\Phi_{s+1}$ is obtained as the time-1 map $X_{F_s}^{t}|_{t=1}$ of the Hamiltonian
vector field $X_{F_s}$ with $F_s=F_{0,s}+F_{1,s}$.
Using (\ref{M7}) we get
\begin{eqnarray}\label{022410}
\left\|F_{i,s}\right\|_{\rho_s}^{+}\leq {\gamma}^{-1} \epsilon_s^{-0.01}\left\|R_{i,s}\right\|_{\rho_s}^{+},\qquad i=0,1.
\end{eqnarray}
Combining (\ref{200}), (\ref{201}), (\ref{022410}) and Lemma \ref{051301}, we get
\begin{equation}\label{Liu2}
\left\|F_{0,s}\right\|_{\rho_s+\delta_s}\leq\frac{64}{e^2\delta_s^2}\left\|F_{0,s}\right\|_{\rho_s}^{+}\leq \epsilon_s^{0.95},
\end{equation}
and
\begin{equation}\label{Liu2}
\left\|F_{1,s}\right\|_{\rho_s+\delta_s}\leq\frac{64}{e^2\delta_s^2}\left\|F_{1,s}\right\|_{\rho_s}^{+}\leq \epsilon_s^{0.55},
\end{equation}
Hence one has
\begin{equation}\label{022501}
\left\|F_{s}\right\|_{\rho_s+\delta_s}\leq \epsilon_s^{0.55}
\end{equation}
and
\begin{equation}\label{Liu4}
\sup_{q\in D_s}\|X_{F_s}\|_{\sigma,\infty}\leq\epsilon_{s}^{0.52},
\end{equation}
which follows from (\ref{050907}) in Lemma \ref{063004}.
Since $$\epsilon_{s}^{0.52}\ll\frac{1}{\pi^2(s+1)^2}=d_{s+1}-d_s,$$ we have $\Phi_{s+1}:D_{s+1}\rightarrow D_{s}$ with
\begin{equation}\label{Liu5}
\left\|\Phi_{s+1}-id\right\|_{\sigma,\infty}\leq\sup_{q\in D_s}\|X_{F_s}\|_{\sigma,\infty}<\epsilon_{s}^{0.5},
\end{equation}
which is the estimate (\ref{203}). Moreover, from (\ref{Liu4}) we get
\begin{equation}\label{Liu6}
\sup_{q\in D_s}\left\|DX_{F_s}-Id\right\|_{\sigma,\infty}\leq\frac{1}{d_s}\epsilon_{s}^{0.52}\ll\epsilon_{s}^{0.5},
\end{equation}
and thus the estimate (\ref{204}) follows.

Now we will estimate the term $R_{s+1}$. Recalling (\ref{051403}), one has $R_{s+1}=R_{2,s}+\mathcal{R}_s$, where
\begin{equation*}
\mathcal{R}_s=\int_{0}^1\{(1-t)\{N_s,F_s\}+R_{0,s}+R_{1,s},F_s\}\circ X_{F_s}^{t}dt+\int_0^1\{R_{2,s},F_s\}\circ X_{F_s}^tdt.
\end{equation*}
Write
\begin{equation*}
R_{s+1}=R_{0,s+1}+R_{1,s+1}+R_{2,s+1},
\end{equation*}
and write $\mathcal{R}_s$ as
\begin{eqnarray}
\mathcal{R}_s
&=&\label{051510}\int_0^1\{R_{0,s},F_s\}\circ X_{F_s}^tdt\\
&&\label{051511}+\int_0^1\{R_{1,s},F_s\}\circ X_{F_s}^tdt\\
&&\label{051513}+\int_0^1\{R_{2,s},F_s\}\circ X_{F_s}^tdt\\
&&\label{051520}+\int_0^1(1-t)\{\{N_s,F_s\},F_s\}\circ X_{F_s}^tdt.
\end{eqnarray}
Firstly note that the term $(\ref{051510})$ contributes to $R_{0,s+1},R_{1,s+1},R_{2,s+1}$ and we get
\begin{eqnarray}
&&\nonumber\left|\left|\int_0^1\{R_{0,s},F_s\}\circ X_{F_s}^tdt\right|\right|_{\rho_s+2\delta_s}
\nonumber\\ \nonumber&=&\left|\left|\sum_{n\geq1}\frac{1}{n!}\underbrace{\left\{\cdots\left\{R_{0,s},{F_s}\right\},\cdots,{F_s}\right\}}_{n-\mbox{fold}}\right|\right|_{\rho_s+2\delta_s}\\
\nonumber&\leq& \frac{2e}{\delta_s}\exp\left\{\frac{2000}{\delta_s}\cdot\exp\left\{\left(\frac{200}{\delta_s}\right)^{\frac1{\sigma-1}}\right\}\right\}||{F_s}||_{\rho_s+\delta_s}||R_{0,s}||_{\rho_s+\delta_s}\ \\
&&\nonumber \mbox{(following the proof of Lemma \ref{E1})}\\
&\leq&\epsilon_s^{1.54}\label{051203}
\end{eqnarray}
where using (\ref{200}), (\ref{022501}) and the fact
\begin{eqnarray*}
\frac{2e}{\delta_s}\exp\left\{\frac{2000}{\delta_s}\cdot\exp\left\{\left(\frac{200}{\delta_s}\right)^{\frac1{\sigma-1}}\right\}\right\}
\leq \epsilon_s^{-0.01},
\end{eqnarray*}
which follows from the proof of (\ref{022402}).
Consequently from (\ref{N6}) one has
\begin{eqnarray}
&&\nonumber\left|\left|\int_0^1\{R_{0,s},F_s\}\circ X_{F_s}^tdt\right|\right|_{\rho_s+3\delta_s}^{+}\\
\nonumber&\leq&\exp\left\{3\left(\frac6{\delta_s}\right)^{\frac1{\sigma-1}}\cdot \exp\left\{\left(\frac6{\delta_s}\right)^{\frac1{\sigma}}\right\} \right\}\left|\left|\int_0^1\{R_{0,s},F_s\}\circ X_{F_s}^tdt\right|\right|_{\rho_s+2\delta_s}\\
&\leq&\epsilon_s^{1.53}\label{N11},
\end{eqnarray}
where using (\ref{051203}) and noting that
\begin{equation*}
\exp\left\{3\left(\frac6{\delta_s}\right)^{\frac1{\sigma-1}}\cdot \exp\left\{\left(\frac6{\delta_s}\right)^{\frac1{\sigma}}\right\} \right\}\leq \epsilon_s^{-0.01},
\end{equation*}which follows from the proof of (\ref{022402}).

Secondly, we consider the term (\ref{051511}) and write
\begin{eqnarray}
(\ref{051511})&=&
\nonumber\sum_{n\geq1}\frac{1}{n!}\underbrace{\{\cdots\{R_{1,s},{F_s}\},{F_s},\cdots,{F_s}\}}_{n-\mbox{fold}}\\
&=&\label{051505}\sum_{n\geq1}\frac{1}{n!}\underbrace{\{\cdots\{R_{1,s},{F}_{0,s}\},{F_s},\cdots,{F_s}\}}_{(n-1)-\mbox{fold}}\\
&&\label{051512}+\{R_{1,s},F_{1,s}\}\\
&&\label{051506}+\sum_{n\geq2}\frac{1}{n!}\underbrace{\{\cdots\{R_{1,s},F_{1,s}\},{F_s},\cdots,{F_s}\}}_{(n-1)-\mbox{fold}}.
\end{eqnarray}
Note that (\ref{051505}) and (\ref{051506}) contribute to $R_{0+},R_{1+},R_{2+}$, and (\ref{051512}) contributes to $R_{1+},R_{2+}$.

Moreover, following the proof of (\ref{N11}), one has
\begin{eqnarray}
\label{N12}\left|\left|(\ref{051505})\right|\right|_{\rho_s+3\delta_s}^{+},\left|\left|(\ref{051506})\right|\right|_{\rho_s+3\delta_s}^{+}
&\leq& \epsilon_{s}^{1.53},\\
\left|\left|(\ref{051512})\right|\right|_{\rho_s+3\delta_s}^{+}&\leq&\epsilon_s^{1.1}.
\end{eqnarray}
Thirdly, we consider the term (\ref{051513}) and write
\begin{eqnarray}
(\ref{051513})&=&\nonumber\sum_{n\geq1}\frac{1}{n!}\underbrace{\{\cdots\{R_{2,s},{F_s}\},{F_s},\cdots,{F_s}\}}_{n-\mbox{fold}}\\
&=&\label{051514}\{R_{2,s},F_{0,s}\}\\
&&+\label{051515}\{R_{2,s},F_{1,s}\}\\
&&+\label{051516}\sum_{n\geq2}\frac{1}{n!}\underbrace{\{\cdots\{R_{2,s},{F}_{0,s}\},{F_s},\cdots,{F_s}\}}_{(n-1)-\mbox{fold}}\\
&&\label{051517}+\{\{R_{1,s},F_{1,s}\},F_s\}\\
&&\label{051518}+\sum_{n\geq3}\frac{1}{n!}\underbrace{\{\cdots\{R_{2,s},F_{1,s}\},{F_s},\cdots,{F_s}\}}_{(n-1)-\mbox{fold}}.
\end{eqnarray}
Note that (\ref{051514}) and (\ref{051517}) contribute to $R_{1+},R_{2+}$, (\ref{051516}) and (\ref{051518}) contribute to $R_{0+},R_{1+},R_{2+}$, and (\ref{051515}) contributes to $R_{2+}$.

Similarly, one has
\begin{eqnarray}
\label{N15}||(\ref{051514})||_{\rho_s+3\delta_s}^{+},\left|\left|(\ref{051517})\right|\right|_{\rho_s+3\delta_s}^{+}&\leq& \epsilon_s^{0.93},\\
\label{N16}\left|\left|(\ref{051516})\right|\right|_{\rho_s+3\delta_s}^{+},\left|\left|(\ref{051518})\right|\right|_{\rho_s+3\delta_s}^{+}
&\leq& \epsilon_s^{1.53}
\end{eqnarray}
and
\begin{equation}
\label{N17}||(\ref{051515})||_{\rho_s+3\delta_s}^{+}
\leq\epsilon_s^{0.5}.
\end{equation}
Finally, we consider the term (\ref{051520}) and
write
\begin{eqnarray}
\nonumber(\ref{051520})&=&\sum_{n\geq2}\frac{1}{n!}\underbrace{\{\cdots\{N_s,{F_s}\},{F_s},\cdots,{F_s}\}}_{n-\mbox{fold}}\\
&=&\nonumber\sum_{n\geq2}\frac{1}{n!}\underbrace{\{\cdots\{-R_{0,s}-R_{1,s}+[R_{0,s}]+[R_{1,s}]\},{F},\cdots,{F}\}}_{(n-1)-\mbox{fold}}\qquad \\
&&\nonumber\mbox{(in view of (\ref{4.27}))}\\
&=&\label{051530}\sum_{n\geq2}\frac{1}{n!}\underbrace{\{\cdots\{-R_{0,s}+[R_{0,s}],{F_s}\},{F_s},\cdots,{F_s}\}}_{(n-1)-\mbox{fold}}\\
&&+\label{051531}\sum_{n\geq2}\frac{1}{n!}\underbrace{\{\cdots\{-R_{1,s}+[R_{1,s}],{F}_{0,s}\},{F_s},\cdots,{F_s}\}}_{(n-2)-\mbox{fold}}\\
&&\label{051532}+\{-R_{1,s}+[R_{1,s}],F_{1,s}\}\\
&&\label{051533}+\sum_{n\geq3}\frac{1}{n!}\underbrace{\{\cdots\{-R_{1,s}+[R_{1,s}],F_{1,s}\},{F_s},\cdots,{F_s}\}}_{(n-1)-\mbox{fold}}.
\end{eqnarray}

Note that (\ref{051530}),  (\ref{051531})  and (\ref{051533}) contribute to $R_{0+},R_{1+},R_{2+}$, and (\ref{051532}) contributes to $R_{1+},R_{2+}$.

Moreover, one has
\begin{eqnarray}
\left|\left|(\ref{051530})\right|\right|_{\rho_s+3\delta_s}^{+},\left|\left|(\ref{051531})\right|\right|_{\rho_s+3\delta_s}^{+},\left|\left|(\ref{051533})\right|\right|_{\rho_s+3\delta_s}^{+}
&\leq&\epsilon_s^{1.53}
\end{eqnarray}
and
\begin{equation}
\left|\left|(\ref{051532})\right|\right|_{\rho_s+3\delta_s}^{+}\leq\epsilon_s^{0.93}.
\end{equation}
Consequently we get
\begin{eqnarray*}
||R_{0,s+1}||_{\rho_{s+1}}^{+}
&\leq&10\epsilon_s^{1.53}\leq \epsilon_{s+1},\\
||R_{1,s+1}||_{\rho_{s+1}}^{+}
&\leq& 10\epsilon_s^{0.93}\leq \epsilon_{s+1}^{0.6}
\end{eqnarray*}
and
\begin{eqnarray*}
||R_{2,s+1}||_{\rho_{s+1}}^{+}
\leq(1+d_s)\epsilon_0+\epsilon_s^{0.5}\leq(1+d_{s+1})\epsilon_0,
\end{eqnarray*}
which are just the assumptions (\ref{200})-(\ref{202}) at stage $s+1$.

In view of (\ref{051402}), the new normal form $N_{s+1}$ is given by
\begin{equation}
N_{s+1}=N_s+[R_{0,s+1}]+[R_{1,s+1}].
\end{equation}Note that $[R_{0,s}]$ (by (\ref{051501})) is a constant which does not affect the Hamiltonian vector field. Moreover, in view of (\ref{051502}), we denote by
\begin{equation}
\omega_{m,s}=m^2+\widetilde V_{m,s}+\sum_{a\in\mathbb{N}^{\mathbb{Z}}}B_{a00}^{(m)}\mathcal{M}_{a00},
\end{equation}
where the term $$\sum_{a\in\mathbb{N}^{\mathbb{Z}}}B_{a00}^{(m)}\mathcal{M}_{a00}$$ is the so-called frequency shift which will be estimated below.
For any $m\in\mathbb{Z}$, one has
\begin{eqnarray}
\nonumber\left|B^{(m)}_{a00}\right|
\nonumber&\leq& \left\|R_{1,s+1}\right\|_{\rho_{s+1}}^+e^{2\rho_{s+1}\left(\sum_{n\in\mathbb{Z}}a_n\ln^{\sigma}\lfloor n\rfloor+\ln^{\sigma}\lfloor m\rfloor-\ln^{\sigma}\lfloor n_1^{*}\rfloor\right)}\\
\label{M12}&<&\epsilon_{s+1}e^{2\rho_{s+1}\left(\sum_{n\in\mathbb{Z}}a_n\ln^{\sigma}\lfloor n\rfloor+\ln^{\sigma}\lfloor m\rfloor-\ln^{\sigma}\lfloor n_1^{*}\rfloor\right)}.
\end{eqnarray}
In view of (\ref{M13}) and (\ref{M12}) ,
we obtain
\begin{eqnarray}
&&\nonumber\left|\sum_{a\in\mathbb{N}^{\mathbb{Z}}}B^{(m)}_{a00}\mathcal{M}_{a00}\right|\\
\nonumber&\leq & \epsilon_{s+1}\sum_{a\in\mathbb{N}^{\mathbb{Z}}}e^{2\rho_{s+1}\left(\sum_{n\in\mathbb{Z}}a_n\ln^{\sigma}\lfloor n\rfloor+\ln^{\sigma}\lfloor m\rfloor-\ln^{\sigma}\lfloor n_1^{*}\rfloor\right)}\cdot e^{-\sum_{n\in\mathbb{Z}}2 a_n\ln^{\sigma}\lfloor n\rfloor}\\
\nonumber&\leq& \epsilon_{s+1}\sum_{a\in\mathbb{N}^{\mathbb{Z}}}e^{-\left(\sum_{n\in\mathbb{Z}}a_n\ln^{\sigma}\lfloor n\rfloor\right)}\qquad\qquad\qquad \mbox{(in view of $\lfloor m\rfloor\leq \lfloor n_1^*\rfloor$ and $2\rho_{s+1}<1$)}\\
\nonumber&\leq& \epsilon_{s+1}\prod_{n\in\mathbb{Z}}\left(1-e^{- \ln^{\sigma}\lfloor n\rfloor}\right)^{-1} \qquad\qquad\qquad \mbox{(by Lemma \ref{a3})}\\
\label{M15}&\leq&\epsilon_{s+1}\exp\left\{{18}\cdot \exp\left\{4^{\frac1{\sigma-1}}\right\}\right\},
\end{eqnarray}
where the last inequality is based on Lemma \ref{a5}.

Next,
if $V\in \mathcal{C}_{\frac{\eta_s}{2}}\left(V_s^*\right)$, by using Cauchy's estimate implies
\begin{eqnarray}
\nonumber\sum_{n\in\mathbb{Z}}\left|\frac{\partial \widetilde{V}_{m,s}}{\partial V_n}(V)\right|
\nonumber&\leq& \frac{2}{\eta_s}\left\|\widetilde{V}_s\right\|_\infty\\
\label{M11}&<&10 \eta_s^{-1}\ \ \mbox{(since $\left\|\widetilde{V}_s\right\|_\infty\leq 2 $)} \ \mbox{for all $m$},
\end{eqnarray}
and let $X\in \mathcal{C}_{\frac{1}{10}\lambda_s\eta_s}\left(V_s^*\right)$, then
\begin{eqnarray*}
\left\|\widetilde{V}_s(X)-\omega\right\|_{\infty}
&=&\left\|\widetilde{V}_s(X)-\widetilde{V}_s\left(V_s^*\right)\right\|_{\infty}\\
& \leq&\sup_{\mathcal{C}_{\frac{1}{10}\lambda_s\eta_s}\left(V_s^*\right)}\left|\left|\frac{\partial \widetilde{V}_s}{\partial V}\right|\right|_{l^{\infty}\rightarrow l^{\infty}}\cdot\left\|X-V_s^*\right\|_{\infty}\\
&<&10 \eta_s^{-1}\cdot\frac{1}{10}\lambda_s\eta_s\qquad \qquad  \ \mbox{(in view of (\ref{M11}))}\\
&=&\lambda_s,
\end{eqnarray*}
that is
\begin{equation*}
\widetilde{V}_s\left(\mathcal{C}_{\frac{1}{10}\lambda_s\eta_s}\left(V_s^*\right)\right)\subseteq \mathcal{C}_{\lambda_s}\left(\omega\right).
\end{equation*}
By (\ref{M15}), we have
\begin{eqnarray}
\left|\widetilde{V}_{m,s+1}-\widetilde{V}_{m,s}\right|
<\epsilon_{s+1}\cdot\exp\left\{{18}\cdot \exp\left\{4^{\frac1{\sigma-1}}\right\}\right\}
\label{M16}<0.9\epsilon_{s+1},
\end{eqnarray}
which verifies (\ref{206}). Further applying Cauchy's estimate on $\mathcal{C}_{\lambda_s\eta_s}\left(V_s^*\right)$, one gets
\begin{eqnarray}\label{633}
\sum_{n\in\mathbb{Z}}\left|\frac{\partial \widetilde{V}_{m,s+1}}{\partial V_n}-\frac{\partial \widetilde{V}_{m,s}}{\partial V_n}\right|
\leq\frac{\left\|\widetilde{V}_{s+1}-\widetilde{V}_{s}\right\|_\infty}{\lambda_s\eta_s}
\label{M17}\leq\frac{0.9\epsilon_{s+1}}{\lambda_s\eta_s}.
\end{eqnarray}
Since
\begin{equation*}
\eta_{s+1}=\frac{1}{20}\lambda_s\eta_s,
\end{equation*}
hence one has
\begin{eqnarray}
\nonumber\lambda_s\eta_{s}&=&20\prod_{i=0}^{s}\left(\frac1{20}\lambda_i\right)\\
\nonumber&=&20\prod_{i=0}^{s}\left(\frac1{20}\epsilon_i^{0.01}\right)\\
\nonumber&\geq&20\prod_{i=0}^{s}\epsilon_i^{0.02}\\
\label{M19}&\geq&20\epsilon_{0}^{0.06(\frac{3}{2})^{s}}.
\end{eqnarray}
On $ \mathcal{C}_{\frac{1}{10}\lambda_s\eta_s}\left(V_s^*\right)$ and for any $m\in\mathbb{Z}$, we deduce from (\ref{M17}), (\ref{M19}) and the assumption (\ref{199}) that
\begin{eqnarray*}
\sum_{n\in\mathbb{Z}}\left|\frac{\partial \widetilde{V}_{m,s+1}}{\partial V_n}-\delta_{mn}\right|
&\leq&\sum_{n\in\mathbb{Z}}\left|\frac{\partial \widetilde{V}_{m,s+1}}{\partial V_n}-\frac{\partial \widetilde{V}_{m,s}}{\partial V_n}\right|+\sum_{n\in\mathbb{Z}}\left|\frac{\partial \widetilde{V}_{m,s}}{\partial V_n}-\delta_{mn}\right|\\
&\leq&\epsilon_{0}^{(\frac{1}{20})(\frac{3}{2})^{s+1}}+d_s\epsilon_{0}^{\frac{1}{10}}\\
&<&d_{s+1}\epsilon_{0}^{\frac{1}{10}},
\end{eqnarray*}
and consequently
\begin{equation}\label{M20}
\left|\left|\frac{\partial \widetilde{V}_{s+1}}{{\partial V}}-Id\right|\right|_{l^{\infty}\rightarrow l^{\infty}}<d_{s+1}\epsilon_{0}^{\frac{1}{10}},
\end{equation}
which verifies (\ref{199}) for $s+1$.

Finally, we will freeze $\omega$ by invoking an inverse function theorem. Consider the following functional equation
\begin{equation}\label{M21}
\widetilde{V}_{s+1}\left(X\right)=\omega, \qquad  X\in \mathcal{C}_{\frac{1}{10}\lambda_s\eta_s}\left(V_s^*\right),
\end{equation}
from (\ref{M20}) and the standard inverse function theorem implies (\ref{M21}) having a solution $V_{s+1}^*$, which verifies (\ref{198}) for $s+1$. Rewriting (\ref{M21}) as
\begin{equation}\label{M22}
V_{s+1}^*-V_s^*=\left(I-\widetilde{V}_{s+1}\right)\left(V_{s+1}^*\right)-\left(I-\widetilde{V}_{s+1}\right)\left({V_s}^*\right)+\left(\widetilde{V}_s-\widetilde{V}_{s+1}\right)\left(V_s^*\right),
\end{equation}
and by using (\ref{M16}) and (\ref{M20}) one has
\begin{equation*}\label{M23}
\left\|V_{s+1}^*-V_s^*\right\|_{\infty}\leq \left(1+d_{s+1}\right)\epsilon_{0}^{\frac{1}{10}}\left\|V_{s+1}^*-V_s^*\right\|_{\infty}+0.9\epsilon_{s+1}<2\epsilon_{s+1}\leq  \lambda_s\eta_s,
\end{equation*}where the last inequality is based on (\ref{M19}),
which verifies (\ref{205}) and completes the proof of the iterative lemma.
\end{proof}

\subsection{Convergence}

We are now in a position to prove the convergence. To apply iterative lemma with $s=0$, set
\begin{equation}\label{031940}
V_0^*=\omega,\hspace{12pt}\widetilde{V}_0=id,\hspace{12pt}\epsilon_0=C\epsilon,
\end{equation}
and consequently (\ref{198})-(\ref{202}) with $s=0$ are satisfied. Hence, the iterative lemma applies, and we obtain a decreasing
sequence of domains $D_{s}\times\mathcal{C}_{\eta_{s}}\left(V_{s}^*\right)$ and a sequence of
transformations
\begin{equation*}
\Phi^s=\Phi_1\circ\cdots\circ\Phi_s:\hspace{6pt}D_{s}\times\mathcal{C}_{\eta_{s}}\left(V_{s}^*\right)\rightarrow D_{0}\times\mathcal{C}_{\eta_{0}}(V_{0}),
\end{equation*}
such that $H\circ\Phi^s=N_s+R_s$ for $s\geq1$. Moreover, the
estimates (\ref{203})-(\ref{205}) hold. Thus we can show $V_s^*$ converge to a limit $V^*$ with the estimate
\begin{equation*}
||V^*-\omega||_{\infty}\leq\sum_{s=0}^{\infty}2\epsilon_{s}^{0.5}<\epsilon_{0}^{0.4},
\end{equation*}
and $\Phi^s$ converge uniformly on $D_*\times\{V_*\}$, where $$D_*=\left\{(q_n)_{n\in\mathbb{Z}}:\frac{2}{3}\leq|q_n|e^{\ln^{\sigma}\lfloor n\rfloor}\leq\frac{5}{6}\right\},$$ to $\Phi:D_*\times\{V^*\}\rightarrow D_0$ with the estimates
\begin{eqnarray}
\nonumber&&||\Phi-id||_{\sigma,\infty}\leq \epsilon_{s}^{0.4},\\
\nonumber&&||D\Phi-Id||_{(\sigma,\infty)\rightarrow(\sigma,\infty)}\leq \epsilon_{s}^{0.4}.
\end{eqnarray}
Hence
\begin{equation}\label{060101}
H_*=H\circ\Phi=N_*+R_{2,*},
\end{equation}
where
\begin{equation}
N_*=\sum_{n\in\mathbb{Z}}(n^2+\omega_n)|q_n|^2
\end{equation}
and
\begin{equation}\label{062811}
||R_{2,*}||_{0.1}^{+}\leq\frac{7}{6}\epsilon_0.
\end{equation}
By (\ref{050907}), the Hamiltonian vector field $X_{R_{2,*}}$ is a bounded map from $\mathfrak{H}_{\sigma,\infty}$ into $\mathfrak{H}_{\sigma,\infty}$. Taking \begin{equation}\label{072701}
I_n(0)=\frac{3}{4}e^{-2\ln^{\sigma}\lfloor n\rfloor},
 \end{equation}we get an invariant torus $\mathcal{T}$ with frequency $(n^2+\omega_n)_{n\in\mathbb{Z}}$ for ${X}_{H_*}$.

\subsection{Proof of Theorem \ref{031930}}
\begin{proof}
Then the  Hamiltonian (\ref{L4}) has the form of
\begin{equation*}\label{H}
	H(q,\bar q)=H_{2}(q,\bar q)+ R(q,\bar q)
\end{equation*}
where
\begin{equation*}
H_{2}(q,\bar q)=\sum_{n\in\mathbb{Z}}\left(n^2+V_n\right)\left|q_n\right|^2
\end{equation*}
and
\begin{equation*}
R(q,\bar q)=\epsilon\sum_{a,k,k'\in\mathbb{N}^{\mathbb{Z}}\atop |a|=0,|k|+|k'|=6}B_{akk'}\mathcal{M}_{akk'}.
\end{equation*}
Then one has
\begin{equation*}
\left\|R\right\|_{\rho_0}\leq \left\|R\right\|_{0}\leq C\epsilon:=\epsilon_0,
\end{equation*}
where $C$ is a  positive constant.

Then the assumptions (\ref{031940}) in iterative lemma for $s=0$ hold.  Applying iterative lemma,  $\Phi(\mathcal{T})$ is the desired invariant torus for the Hamiltonian (\ref{L4}) of NLS (\ref{L1}). Moreover, we deduce the torus $\Phi(\mathcal{T})$ is linearly stable from the fact that (\ref{060101}) is a normal form of order 2 around the invariant torus.
\end{proof}
\section{Appendix}
\begin{lem}\label{122203}
Given any $\sigma>1$, there exists a constant $c(\sigma)>1$ depending on $\sigma$ only such that
\begin{equation}\label{122201}
\ln^{\sigma}(x+y)-\ln^{\sigma}x-\frac12\ln^{\sigma}y\leq 0, \quad \mbox{for}\ c(\sigma)\leq y\leq x.
\end{equation}
\end{lem}
\begin{proof}
Write $x=ty$ with $t\geq 1$, and then
\begin{eqnarray*}
&&\ln^{\sigma}(x+y)-\ln^{\sigma}x-\frac12\ln^{\sigma}y\\
&=&\ln^{\sigma}(ty+y)-\ln^{\sigma}(ty)-\frac12\ln^{\sigma}y\\
&=&\ln^{\sigma}y\left(\left(1+\frac{\ln(1+t)}{\ln y}\right)^{\sigma}-\left(1+\frac{\ln t}{\ln y}\right)^{\sigma}-\frac12\right).
\end{eqnarray*}
To prove (\ref{122201}), it suffices to show that
\begin{equation}\label{122202}
\sup_{t\geq 1,y\geq c(\sigma)}\left(\left(1+\frac{\ln(1+t)}{\ln y}\right)^{\sigma}-\left(1+\frac{\ln t}{\ln y}\right)^{\sigma}\right)\leq \frac12.
\end{equation}
Using differential mean value theorem, one has
\begin{eqnarray*}
\left(1+\frac{\ln(1+t)}{\ln y}\right)^{\sigma}-\left(1+\frac{\ln t}{\ln y}\right)^{\sigma}
\leq \frac{\sigma}{\ln y}\left(1+\frac{\ln(1+t)}{\ln y}\right)^{\sigma-1}\ln \left(\frac{1+t}{t}\right).
\end{eqnarray*}

\textbf{Case 1.} $\frac{\ln(1+t)}{\ln y}\leq 1.$

Then one has
\begin{equation*}
\frac{\sigma}{\ln y}\left(1+\frac{\ln(1+t)}{\ln y}\right)^{\sigma-1}\ln \left(\frac{1+t}{t}\right)\leq \frac{\sigma\cdot 2^{\sigma-1}}{\ln y}\cdot \ln 2,
\end{equation*}
where using
\begin{equation*}
\ln \left(\frac{1+t}{t}\right)\leq \ln 2.
\end{equation*}
Taking
\begin{equation*}
c_1(\sigma)=\exp\left\{\sigma\cdot 2^{\sigma}\cdot\ln 2 \right\},
\end{equation*}
we finish the proof of (\ref{122202}) in view of $y\geq c_1(\sigma)$.

\textbf{Case 2.} $\frac{\ln(1+t)}{\ln y}>1.$

Then one has
\begin{eqnarray*}
&&\frac{\sigma}{\ln y}\left(1+\frac{\ln(1+t)}{\ln y}\right)^{\sigma-1}\ln \left(\frac{1+t}{t}\right)\\
&\leq& \frac{\sigma\cdot 2^{\sigma-1}}{\ln^{\sigma}y}\cdot \ln^{\sigma-1}(1+t)\cdot \ln\left(\frac{1+t}{t}\right)\\
&\leq&\frac{\sigma\cdot 2^{\sigma-1}}{\ln^{\sigma}y}\cdot 2\left(\frac{\sigma-1}e\right)^{\sigma-1},
\end{eqnarray*}
where noting that
\begin{equation*}
\max_{t\geq 1}\left(\ln^{\sigma-1}(1+t)\cdot \ln\left(\frac{1+t}{t}\right)\right)\leq 2\left(\frac{\sigma-1}e\right)^{\sigma-1}:=c^*(\sigma).
\end{equation*}
Taking
\begin{equation*}
c_2(\sigma)=\exp\left\{{\left(\sigma\cdot 2^{\sigma}\cdot c^*(\sigma)\right)^{\frac1{\sigma}}}\right\},
\end{equation*}
we finish the proof of (\ref{122202}) in view of $y\geq c_2(\sigma)$.

Finally letting
\begin{equation}\label{030402}
c(\sigma)=\max \{c_1(\sigma),c_2(\sigma)\},
\end{equation}
we finish the proof of (\ref{122201}).
\end{proof}
\begin{lem}\label{a1}Let $k=(k_n)_{n\in\mathbb{Z}},k'=(k'_n)_{n\in\mathbb{Z}}\in\mathbb{N}^{\mathbb{Z}}$ with $4\leq |k|+|k'|<\infty$ and $\widetilde V=\left(\widetilde V_n\right)_{n\in\mathbb{Z}}\in\mathbb{R}^{\mathbb{Z}}$ satisfy $\left|\widetilde{V}_n\right|\leq2\ \mbox{for}\ \forall \ n\in\mathbb{Z}$.
Assume further
\begin{equation}
\label{041803}\left|\sum_{n\in\mathbb{Z}}\left(k_n-k_n'\right)\left(n^2+\widetilde V_n\right)\right|\leq1
\end{equation}
and
\begin{equation}
\label{041803'}\sum_{n\in\mathbb{Z}}(k_n-k_n')n=0.
\end{equation}
Then one has
\begin{equation}\label{041809'}
\sum_{n\in\mathbb{Z}}|k_n-k_n'|\ln^{\sigma}\lfloor n\rfloor\leq3 \cdot 4^{\sigma} \sum_{i\geq 3}\ln^{\sigma}\lfloor n_i\rfloor,
\end{equation}
where $(n_i)_{1\leq i\leq m}$ denotes the system \{$|n|$: $n$ is repeated $k_n+k'_n$ times\} which satisfies $|n_1|\geq|n_2|\geq|n_3|\geq\cdots\geq |n_{m}|$ with $m=|k|+|k'|$. For short let $$\sum_{i\geq 3}=\sum_{3\leq i\leq m}.$$
\end{lem}
\begin{proof}
From the definition of $(n_i)_{1\leq i\leq m}$ and (\ref{041803'}), there exist $(\mu_i)_{1\leq i\leq m}$ with $\mu_i\in\{-1,1\}$ such that
\begin{equation}
\label{1605141}\sum_{n\in\mathbb{Z}}(k_n-k_n')n^2=\sum_{i\geq 1}\mu_in_i^2
\end{equation}
\and
\begin{equation}
\label{1605141'}\sum_{n\in\mathbb{Z}}(k_n-k_n')n=\sum_{i\geq 1}\mu_in_i,
\end{equation}
where $$\sum_{i\geq 1}=\sum_{1\leq i\leq m}.$$
In view of (\ref{041803}), (\ref{1605141}) and $\left|\widetilde V_n\right|\leq 2$,
one has
\begin{equation*}
\left|\sum_{i\geq 1}\mu_in_i^2\right|\leq\left|\sum_{n\in\mathbb{Z}}(k_n-k_n')\widetilde V_n\right|+1\leq2m+1,
\end{equation*}which implies
\begin{equation}\label{041804}
\left|n_1^2+\left(\frac{\mu_2}{\mu_1}\right)n_2^2\right|\leq\left(\sum_{i\geq 3}n_i^2\right)+2m+1\leq 2\sum_{i\geq 3}\lfloor n_i\rfloor^2.
\end{equation}
On the other hand, by (\ref{041803'}) and (\ref{1605141'}), we obtain
\begin{equation}\label{041805}
\left|n_1+\left(\frac{\mu_2}{\mu_1}\right)n_2\right|\leq \sum_{i\geq 3}|n_i|\leq \sum_{i\geq 3}\lfloor n_i\rfloor.
\end{equation}

To prove the inequality (\ref{041809'}), we will distinguish two cases:

\textbf{Case. 1.} $\frac{\mu_2}{\mu_1}=-1$.

\textbf{Case. 1.1.} $n_1=n_2$.

Then it is easy to see that
\begin{equation*}
\sum_{n\in\mathbb{Z}}|k_n-k_n'|\ln^{\sigma}\lfloor n\rfloor\leq \sum_{i\geq 3}\ln^{\sigma}\lfloor n_i\rfloor.
\end{equation*}

\textbf{Case. 1.2.} $n_1\neq n_2$.

Then one has
\begin{eqnarray*}\label{041806}
\nonumber|n_1-n_2|+|n_1+n_2|
&\leq&\nonumber|n_1-n_2|+\left|n_1^2-n_2^2\right|\\
& \leq&\nonumber\sum_{i\geq 3}\lfloor n_i\rfloor+2\sum_{i\geq 3}\lfloor n_i\rfloor^2\quad \mbox{(by (\ref{041804}) and (\ref{041805}))}\\
 &\leq&3\sum_{i\geq 3}\lfloor n_i\rfloor^2.
\end{eqnarray*}
Hence one has
\begin{equation*}
|n_1|\leq |n_1-n_2|+|n_1+n_2|\leq 3\sum_{i\geq 3}\lfloor n_i\rfloor^2.
\end{equation*}
and then
\begin{equation*}
\lfloor n_1\rfloor \leq 3\sum_{i\geq 3}\lfloor n_i\rfloor^2.
\end{equation*}
Furthermore the following inequalities hold
\begin{equation*}
\ln\lfloor n_1\rfloor \leq \ln 3+\ln\left(\sum_{i\geq 3}\lfloor n_i\rfloor^2\right)
\end{equation*}
and
\begin{equation}\label{030501}
\ln^{\sigma}\lfloor n_1\rfloor \leq 2^{\sigma}\ln^{\sigma}\left(\sum_{i\geq 3}\lfloor n_i\rfloor^2\right) .
\end{equation}
Using
  (\ref{122201}) in Lemma \ref{122203} again and again, one has
\begin{equation}\label{022202}
\ln^{\sigma}\left(\sum_{i\geq 3}\lfloor n_i\rfloor^2\right) \leq2^{\sigma}\sum_{i\geq 3}\ln^{\sigma}\lfloor n_i\rfloor.
\end{equation}
Finally we have
\begin{eqnarray}
\nonumber\sum_{n\in\mathbb{Z}}|k_n-k_n'|\ln^{\sigma}\lfloor n\rfloor&\leq&\nonumber\sum_{n\in\mathbb{Z}}(k_n+k_n')\ln^{\sigma}\lfloor n\rfloor\\
&=&\nonumber\sum_{i\geq 1}\ln^{\sigma}\lfloor n_i\rfloor\\
&\leq&\label{2016121101}3 \cdot 4^{\sigma} \sum_{i\geq 3}\ln^{\sigma}\lfloor n_i\rfloor,
\end{eqnarray}
where the last inequality is based on (\ref{030501}) and (\ref{022202}).

\textbf{Case. 2.} $\frac{\mu_2}{\mu_1}=1$.

In view of (\ref{041804}), one has
\begin{equation*}
n_1^2+n_2^2\leq 2\sum_{i\geq3}\lfloor n_i\rfloor^{2},
\end{equation*}which implies
\begin{equation*}
\lfloor n_1\rfloor\leq 2\sum_{i\geq3}\lfloor n_i\rfloor^2.
\end{equation*}
Following the proof of (\ref{2016121101}), one has
\begin{equation*}
\sum_{n\in\mathbb{Z}}|k_n-k_n'|\ln^{\sigma}\lfloor n\rfloor\leq
3 \cdot 4^{\sigma} \sum_{i\geq3}\ln^{\sigma}\lfloor n_i\rfloor.
\end{equation*}

\end{proof}
\begin{lem}\label{a3}
Assuming $\sigma>1$ and $\delta\in(0,1)$, then we have the following inequality
\begin{equation}\label{041809}
\sum_{a\in\mathbb{N}^{\mathbb{Z}}}e^{-\delta\sum_{n\in\mathbb{Z}}a_{n}\ln^{\sigma}\lfloor n\rfloor}\leq\prod_{n\in\mathbb{Z}}\frac{1}{1-e^{-\delta \ln^{\sigma}\lfloor n\rfloor}}.
\end{equation}
\end{lem}
\begin{proof}
By a direct calculation, one has
$$\sum\limits_{a\in\mathbb{N}^{\mathbb{Z}}}e^{-\delta\sum_{n\in\mathbb{Z}}a_{n}\ln^{\sigma}\lfloor n\rfloor}\leq \prod\limits_{n\in\mathbb{Z}}\left(\sum\limits_{a_n\in\mathbb{N}}e^{-\delta a_{n}\ln^{\sigma}\lfloor n\rfloor}\right)
=\prod\limits_{n\in\mathbb{Z}}\frac{1}{1-e^{-\delta \ln^{\sigma}\lfloor n\rfloor}}.$$
\end{proof}
\begin{lem}\label{8.6}For $\sigma>1$ and $\delta\in(0,1)$, let $$f_{\sigma,\delta}(x)=e^{-\delta x^{\sigma}+x},$$ then we have
\begin{equation}\label{042805}
\max_{x\geq1}f_{\sigma,\delta}(x)\leq \exp\left\{\left(\frac{1}{\delta\sigma}\right)^{\frac1{\sigma-1}}\right\}.
\end{equation}
\end{lem}
\begin{proof}
Firstly we have
\begin{equation*}
f_{\sigma,\delta}'(x)=e^{-\delta x^{\sigma}+x}\left(-\delta\sigma x^{\sigma-1}+1\right).
\end{equation*}
Hence
\begin{equation*}
f_{\sigma,\delta}'(x)=0\Leftrightarrow x=\left(\frac{1}{\delta\sigma}\right)^{\frac1{\sigma-1}}.
\end{equation*}
Then one has
\begin{equation*}
\max_{x\geq1}f_{\sigma,\delta}(x)\leq f_{\sigma,\delta}\left(\left(\frac{1}{\delta\sigma}\right)^{\frac1{\sigma-1}}\right)\leq \exp\left\{\left(\frac{1}{\delta\sigma}\right)^{\frac1{\sigma-1}}\right\}.
\end{equation*}
\end{proof}
\begin{lem}\label{8.6}For $p\geq 1$ and $\delta\in (0,1)$, let $$g_{p,\delta}(x)=x^{p}e^{-\delta x},$$ then we have
\begin{equation}\label{042805*}
\max_{x\geq1}g_{p,\delta}(x)\leq \left(\frac{p}{e\delta}\right)^p.
\end{equation}
\end{lem}
\begin{proof}
Firstly we have
\begin{equation*}
g_{p,\delta}'(x)=px^{p-1}e^{-\delta x}-\delta x^{p}e^{-\delta x}.
\end{equation*}
Hence
\begin{equation*}
g_{p,\delta}'(x)=0\Leftrightarrow x=\frac p{\delta}.
\end{equation*}
Then one has
\begin{equation*}\label{042608*}
\max_{x\geq1}g_{p,\delta}(x)\leq g_{p,\delta}\left(\frac{p}\delta\right)= \left(\frac{p}{e\delta}\right)^p.
\end{equation*}
\end{proof}
\begin{lem}\label{lem2}
Assuming $\sigma>1$ and $\delta\in(0,1)$, then we have
\begin{equation}\label{0418011}
\sum_{n\geq 3}e^{-\delta \ln^{\sigma} n}\leq \frac{3}\delta\cdot \exp\left\{\left(\frac2{\delta\sigma}\right)^{\frac1{\sigma-1}}\right\}.
\end{equation}
\end{lem}
\begin{proof}
Obviously, we have
\begin{eqnarray*}
\label{e}&&\sum_{n\geq 3}e^{-\delta \ln^{\sigma} n}\\
 &\leq& \int_{1}^{+\infty}e^{-\delta \ln^{\sigma} x}\mathrm{d}x\\
&=&\int_{0}^{+\infty}e^{-\delta y^{\sigma}+y}\mathrm{d}y\\
&\leq&\exp\left\{\left(\frac2{\delta\sigma}\right)^{\frac1{\sigma-1}}\right\}\cdot\int_{0}^{+\infty}e^{-\frac12\delta y^{\sigma}}\mathrm{d}y\qquad \mbox{(by (\ref{042805}))}\\
&\leq&\exp\left\{\left(\frac2{\delta\sigma}\right)^{\frac1{\sigma-1}}\right\}\cdot\left(1+\int_{1}^{+\infty}e^{-\frac12\delta y}\mathrm{d}y\right)\\
&\leq &\frac{3}\delta\cdot \exp\left\{\left(\frac2{\delta\sigma}\right)^{\frac1{\sigma-1}}\right\}.
\end{eqnarray*}

\end{proof}
\begin{lem}\label{a5}
Assuming $\sigma>1$ and $\delta\in(0,3-2\sqrt{2})$, then we have
\begin{equation}\label{122401}
\prod_{n\in\mathbb{Z}}\frac{1}{1-e^{-\delta \ln^{\sigma}\lfloor n\rfloor}}\leq\exp\left\{\frac{18}\delta\cdot \exp\left\{\left(\frac4{\delta}\right)^{\frac1{\sigma-1}}\right\}\right\}.
\end{equation}
\end{lem}
\begin{proof}
If $x\in(0,3-2\sqrt{2})$, one has
\begin{equation}\label{042506}\frac{1}{1-e^{-x}}\leq \frac{1}{x^{2}}
\end{equation}
and
\begin{equation}\label{022208}
\ln\left(\frac{1}{1-x}\right)\leq\sqrt{x}.
\end{equation}
Let \begin{equation*}\label{ntheta}
N=\exp\left\{\left(\frac{\ln\left(3+2\sqrt{2}\right)}{\delta}\right)^{\frac1\sigma}\right\}.
\end{equation*}If $|n|>N$, then one has
\begin{equation*}
-\delta\ln^{\sigma}| n|\leq \ln\left(3-2\sqrt{2}\right),
\end{equation*}
which implies
\begin{equation}\label{022207}
e^{-\delta\ln^{\sigma}\lfloor n\rfloor}\leq e^{-\delta\ln^{\sigma}| n|}\leq 3-2\sqrt{2}.
\end{equation}

Hence we have
\begin{eqnarray*}
\nonumber&&\prod\limits_{n\in\mathbb{Z}}\frac{1}{1-e^{-\delta \ln^{\sigma}\lfloor n\rfloor}}
\nonumber\\
&=&\nonumber\left(\prod\limits_{|n|\leq N}\frac{1}{1-e^{-\delta \ln^{\sigma}\lfloor n\rfloor }}\right)\left(\prod\limits_{|n|> N}\frac{1}{1-e^{-\delta \ln^{\sigma} \lfloor n\rfloor }}\right)\\
&\leq&\nonumber\left(\prod\limits_{|n|\leq N}\frac{1}{1-e^{-\delta }}\right)\left(\prod\limits_{|n|> N}\frac{1}{1-e^{-\delta \ln^{\sigma} \lfloor n\rfloor }}\right).\end{eqnarray*}
On one hand, using (\ref{042506}) we have
\begin{eqnarray}\label{022210}
\prod\limits_{|n|\leq N}\frac{1}{1-e^{-\delta }}\leq \left(\frac{1}{\delta}\right)^{6N}=\exp\left\{\ln\left(\frac1{\delta}\right)\cdot 6\exp\left\{\left(\frac{\ln\left(3+2\sqrt{2}\right)}{\delta}\right)^{\frac1\sigma}\right\}\right\}.
\end{eqnarray}
On the other hand, by (\ref{022208}) and (\ref{022207}) one has
\begin{eqnarray}
\nonumber\prod\limits_{|n|> N}\frac{1}{1-e^{-\delta \ln^{\sigma} \lfloor n\rfloor }}
&=&\nonumber \exp\left\{\sum_{|n|>N}\ln\left(\frac{1}{1-e^{-\delta \ln^{\sigma} \lfloor n\rfloor }}\right)\right\}\\
&\leq&\exp\left\{\sum_{|n|>N}e^{-\frac12\delta \ln^{\sigma} \lfloor n\rfloor }\right\}\nonumber\\
&\leq&\exp\left\{\frac{12}\delta\cdot \exp\left\{\left(\frac4{\delta}\right)^{\frac1{\sigma-1}}\right\}\right\}\label{022211},
\end{eqnarray}
where the last inequality is based on (\ref{0418011}) in Lemma \ref{lem2}and  using $\sigma>1$.

Combing (\ref{022210}) and (\ref{022211}), we finish the proof of (\ref{122401}).

\end{proof}
\begin{lem}Assuming $\sigma>1$, $\delta\in(0,1)$, $p=1,2$ and $a=(a_n)_{n\in\mathbb{Z}}\in\mathbb{N}^{\mathbb{Z}}$, then we have
\begin{equation}\label{042807}
\prod_{n\in\mathbb{Z}}\left(1+a_n^p\right)e^{-2\delta a_n\ln^{\sigma} \lfloor n\rfloor}\leq \exp\left\{3p\left(\frac{p}{\delta}\right)^{\frac 1{\sigma-1}}\cdot\exp\left\{\left(\frac1\delta\right)^{\frac1\sigma}\right\}
\right\}.
\end{equation}
\end{lem}
\begin{proof}
Let
\begin{equation*}
f_{p,\sigma,\delta}(x)=x^{p}e^{-\delta\ln^{\sigma}x},
\end{equation*}
and one has
\begin{equation}\label{030102}
\max_{x\geq 1}f_{p,\sigma,\delta}(x)\leq \exp\left\{p\left(\frac{p}{\delta\sigma}\right)^{\frac1{\sigma-1}}\right\}.
\end{equation}
In fact
\begin{equation*}
f'_{p,\sigma,\delta}(x)=px^{p-1}e^{-\delta\ln^{\sigma}x}
+x^{p}e^{-\delta\ln^{\sigma}x}\left(-\delta\sigma\ln^{\sigma-1}x\cdot x^{-1}\right)
\end{equation*}
and
\begin{equation*}
f'_{p,\sigma,\delta}(x)=0\Leftrightarrow x=\exp\left\{\left(\frac{p}{\delta\sigma}\right)^{\frac1{\sigma-1}}\right\}.
\end{equation*}
Hence
\begin{equation*}
\max_{x\geq 1}f_{p,\sigma,\delta}(x)\leq f\left(\exp\left\{\left(\frac{p}{\delta\sigma}\right)^{\frac1{\sigma-1}}\right\}\right)\leq \left(\exp\left\{\left(\frac{p}{\delta\sigma}\right)^{\frac1{\sigma-1}}\right\}\right)^p= \exp\left\{p\left(\frac{p}{\delta\sigma}\right)^{\frac1{\sigma-1}}\right\}.
\end{equation*}
Note that
$$\prod_{n\in\mathbb{Z}}\left(1+a_n^2\right)e^{-2\delta a_n\ln^{\sigma}\lfloor n\rfloor}=\prod_{n\in\mathbb{Z}:a_n\geq1}\left(1+a_n^2\right)e^{-2\delta a_n\ln^{\sigma}\lfloor n\rfloor}.$$
Then we can assume $a_n\geq 1\ \mbox{for}\ \forall\  n\in\mathbb{Z}$ in what follows. Thus one has
\begin{eqnarray}
\nonumber&&\left(1+a_n^p\right)e^{-2\delta a_n\ln^{\sigma}\lfloor n\rfloor}\\
\nonumber&\leq&\left( 2a_n\right)^pe^{-2\delta a_n\ln^{\sigma}\lfloor n\rfloor}\\
\nonumber&=&\frac{1}{\ln^{\sigma p}\lfloor n\rfloor}\cdot\left(2a_n\ln^{\sigma}\lfloor n\rfloor\right)^pe^{-2\delta a_n\ln^{\sigma}\lfloor n\rfloor}\\
\label{160516}&\leq&\exp\left\{p\left(\frac{p}{\delta\sigma}\right)^{\frac1{\sigma-1}}\right\},
\end{eqnarray}
where the last inequality is based on (\ref{030102}) and using the fact that
\begin{equation*}
\frac{1}{\ln^{\sigma p}\lfloor n\rfloor}\leq 1.
\end{equation*}
On the other hand, for $p=1,2$ one has
\begin{equation}\label{042806}
(1+a_n^p)e^{-2\delta a_n\ln^{\sigma} \lfloor n\rfloor}\leq 1
\end{equation}
when $\ln^{\sigma} \lfloor n\rfloor\geq \delta^{-1}$.

Therefore, we have
\begin{eqnarray*}
&&\prod_{n\in\mathbb{Z}}\left(1+a_n^p\right)e^{-2\delta a_n\ln^{\sigma}\lfloor n\rfloor}\\
&=&\left(\prod_{\ln^{\sigma} \lfloor n\rfloor< \delta^{-1}}\left(1+a_n^p\right)e^{-2\delta a_n\ln^{\sigma}\lfloor n\rfloor}\right)\left(\prod_{\ln^{\sigma} \lfloor n\rfloor\geq \delta^{-1}}\left(1+a_n^p\right)e^{-2\delta a_n\ln^{\sigma}\lfloor n\rfloor}\right)\\
&\leq&\prod_{\ln^{\sigma} \lfloor n\rfloor<\delta^{-1}}\left(1+a_n^p\right)e^{-2\delta a_n\ln^{\sigma}\lfloor n\rfloor}\qquad \mbox{(in view of (\ref{042806}))}\\
&\leq&\prod_{\ln^{\sigma} \lfloor n\rfloor<\delta^{-1}}\exp\left\{p\left(\frac{p}{\sigma\delta}\right)^{\frac1{\sigma-1}}\right\}\qquad \ \mbox{(in view of (\ref{160516}))}\\
&\leq&\left(\exp\left\{p\left(\frac{p}{\delta\sigma}\right)^{\frac1{\sigma-1}}\right\}\right)^{3\exp\left\{\left(\frac1\delta\right)^{\frac1\sigma}\right\}}\ \ \ \\
&=&\exp\left\{3p\left(\frac{p}{\delta\sigma}\right)^{\frac 1{\sigma-1}}\cdot\exp\left\{\left(\frac1\delta\right)^{\frac1\sigma}\right\}
\right\}\\
&\leq&\exp\left\{3p\left(\frac{p}{\delta}\right)^{\frac 1{\sigma-1}}\cdot\exp\left\{\left(\frac1\delta\right)^{\frac1\sigma}\right\}
\right\},
\end{eqnarray*}
where using $\sigma>1$.
\end{proof}

\section*{Acknowledgments}
   The author is supported by NNSFC No. 12071053.

\bibliographystyle{alpha}

\end{document}